\documentclass[smallextended,envcountsect]{svjour3}

\smartqed

\usepackage{graphicx}

\usepackage{latexsym,amsmath,amsfonts,amscd,amssymb,verbatim,hyperref}
\usepackage{color}
\usepackage{cite}
\usepackage[shortlabels]{enumitem}

\def\tto{\;{\lower 1pt \hbox{$\rightarrow$}}\kern -10pt
	\hbox{\raise 2pt \hbox{$\rightarrow$}}\;}

\def\ra{\rangle}
\def\la{\langle}

\def\epsilon{\varepsilon}

\def\R{\Bbb R}

\begin{document}
	
	\title{Solving Indefinite Quadratic Programs by Dynamical Systems: Preliminary Investigations\footnote{Dedicated to Professor Franco Giannessi on the occasion of his 90th birthday.}}
	
	\author{Massimo Pappalardo \and Nguyen Nang Thieu \and Nguyen Dong Yen}
	
	\institute{Massimo Pappalardo \at Dipartimento di Informatica, Universit\`a di Pisa, 56127 Pisa, Italy\\ massimo.pappalardo@unipi.it \and  Nguyen Nang Thieu \at Institute of Mathematics, Vietnam Academy of Science and Technology\\
		18 Hoang Quoc Viet, Hanoi 10307, Vietnam\\
		nnthieu@math.ac.vn\and 
		Nguyen Dong Yen, Corresponding author \at
		Institute of Mathematics, Vietnam Academy of Science and Technology\\
		18 Hoang Quoc Viet, Hanoi 10307, Vietnam\\
		ndyen@math.ac.vn		
	}
	
	\date{Received: date / Accepted: date}
	
	\titlerunning{Solving Indefinite Quadratic Programs by Dynamical Systems}
	
	\authorrunning{M. Pappalardo, N. N. Thieu, N. D. Yen}
	
	\maketitle
	
\begin{abstract} Preliminary results of our investigations on solving indefinite qua\-dra\-tic programs by dynamical systems are given. First, dynamical systems corresponding to two fundamental DC programming algorithms to deal with indefinite quadratic programs are considered. Second, the existence and the uniqueness of the global solution of the dynamical system are proved by using some theorems from nonsmooth analysis and the theory of ordinary differential equations. Third, the strong pseudomonotonicity of the restriction of an affine operator on a closed convex set is analyzed in a special case. Finally, for a parametric indefinite quadratic program related to that special case, convergence of the trajectories of the dynamical system to the Karush-Kuhn-Tucker points is established. The elementary direct proofs in the third and fourth topics would be useful for understanding the meaning and significance of several open problems proposed in this paper. 
\end{abstract}

\keywords{Indefinite qua\-dra\-tic program, DC programming algorithm, affine variational inequality, strong pseudomonotonicity, dynamical system, KKT point set, convergence}

\subclass{90C20 . 90C26 . 47J20 · 49J40 · 49J53 . 49M30 . 34A12}

\section{Introduction}\label{Sect-1}

Our aim in this paper is to show how dynamical systems can be useful for solving indefinite quadratic programs.

\medskip
Consider the \textit{indefinite quadratic programming problem under a geometric constraint}
\begin{eqnarray}\label{QP problem}
	\min\Big\{f(x)=\dfrac{1}{2}x^TQx+q^Tx \mid x\in C\Big\}, \end{eqnarray}
where $Q\in\mathbb R^{n\times n}$ is a symmetric matrix,  $q\in\mathbb  R^n$,  $C\subset\R^n$ is a nonempty closed convex set, and $^T$ signifies the matrix transposition. This problem is well known in optimization theory. Many qualitative properties of~\eqref{QP problem} were established for the case where $C$ is a polyhedral convex set (see, e.g.,~\cite{lty05} and the references therein). Numerical methods to solve the problem in that case were studied in~\cite{CLY_2024,aty2011,PhamDinh_LeThi_Akoa,Tuan_JMAA2015,ye1} (see also the literature cited in those papers). If $C$ is a closed ball with center $a$ and radius $r>0$, i.e., $$C=\bar B(a,r):=\big\{x\in \mathbb R^n\mid \|x-a\|\leq r\big\},$$ then by introducing the new variable $y:=x-a$ one gets the problem on minimizing a linear quadratic function on the ball $\bar B(0,r)$. Thus,~\eqref{QP problem} includes the \textit{trust-region subproblem}~\cite{cgt} as a special case. Note that interesting qualitative properties of the later problem were obtained in~\cite{lty2012,lpr,Qui_Yen_2014}. Numerical methods to solve the problem were considered in~\cite{aty2012,m,PhamDinh_LeThi98,Tuan_Yen_JOGO2013,ye1}.

\medskip
According to the generalized Fermat's rule (see, for example,~\cite[p.~2]{ks80} and~\cite[Theorem~3.1]{lty05}), if $\bar{x}\in C$ is a \textit{local solution} of~\eqref{QP problem}, then 
\begin{equation}\label{fermat_rule}
	\langle Q\bar{x}+q, y-\bar{x}\rangle  \geq 0,\quad\mbox{\rm for all}\;\;y\in C,
\end{equation}

\begin{definition}\label{KKT_point} {\rm A vector $\bar{x}\in C$ is said to be a \textit{Karush-Kuhn-Tucker point} (a KKT point) of~\eqref{QP problem} if condition~\eqref{fermat_rule} is satisfied. The KKT point set of~\eqref{QP problem} is denoted by $C_*$.}
\end{definition}

\begin{definition}\label{AVI} {\rm Let $Q\in\mathbb R^{n\times n}$ be a matrix,  $q\in\mathbb  R^n$, and $C\subset\R^n$ a nonempty closed convex set. The \textit{affine variational inequality with a geometric constraint set} defined by the triple $(Q,q,C)$ is the problem of finding $\bar x\in C$ such that  condition~\eqref{fermat_rule} is fulfilled.}
\end{definition}

In the terminology of Definition~\ref{AVI}, the KKT point set of the indefinite quadratic programming problem under a geometric constraint~\eqref{QP problem} coincides with the solution set of the affine variational inequality  with a geometric constraint set (AVI for brevity) defined by the triple $(Q,q,C)$, where $Q$ is a symmetric matrix. In general, the matrix $Q$ of the AVI given by $(Q,q,C)$ needs not be symmetric.

\medskip
To solve~\eqref{QP problem} by DCA (Difference-of-Convex functions Algorithm)~(see~\cite{PhamDinh_LeThi_AMV97} and \cite{PhamDinh_LeThi98} for more details), one uses the representation $f(x)=f_1(x)-f_2(x)$ for the objective function $f(x)$ with 
\begin{eqnarray*}\label{components_IQP} f_1(x)=\Big[\frac{1}{2}x^TQ_1x+q^Tx\Big]+\delta_C(x),\quad  f_2(x)=\frac{1}{2}x^TQ_2x,
\end{eqnarray*} where $Q_1$ and $Q_2$ are such symmetric positive definite  $n\times n$ matrices that $Q=Q_{1}-Q_{2}$, Here, $\delta_C(x)$ is the \textit{indicator function} of $C$ (i.e., $\delta_C(x)=0$ for $x\in C$ and $\delta_C(x)=\infty$ for $x\notin C$). Denote by $\lambda_1(Q)$ and $\lambda_n(Q)$, respectively, the smallest eigenvalue and the largest eigenvalue of $Q$. Following Pham Dinh et al.~\cite{PhamDinh_LeThi98,PhamDinh_LeThi_Akoa}, one can put

(a) $Q_1=\rho I$, $Q_2=\rho I-Q$, where  $\rho>0$ is a real value such that $\rho >\lambda_n(Q)$;

(b) $Q_1=Q+\rho  I$, $Q_2=\rho I$, where $\rho> 0$ is a real value satisfying the condition $\rho >-\lambda_1(Q)$.

\medskip 
Since an upper bound and a lower bound for the eigenvalues of a symmetric matrix can be easily computed (see~\cite[the estimate~(a) or the estimate~(b) on p.~418]{Stoer_Bulirsch_1980}), one can quickly find a real value $\rho>0$ satisfying $\rho >\lambda_n(Q)$ (resp., $\rho >-\lambda_1(Q)$).

\medskip
The choice made in~(a) leads to the explicit iteration scheme  \begin{eqnarray}\label{Iteration_A}
	x^{k+1}:=P_C\Big(x^k-\frac{1}{\rho}(Qx^k+q)\Big) \quad (k=0,1, 2,\dots)\end{eqnarray} with $x^0\in C$ being an initial point and $P_C(u)$ denoting the metric projection of $u\in\mathbb R^n$ on $C$. Remarkable properties of the DCA sequences in question are given in the next results.
	
\begin{theorem}\label{Tuan's_thm} {\rm (See~\cite[Theorem~2.1]{Tuan_JMAA2015})} If~$C$ is a polyhedral convex set and~\eqref{QP problem} has a global solution, then for each $x^0\in C$, the DCA sequence $\{x^k\}$ constructed by the scheme~\eqref{Iteration_A} converges $R$-linearly to a KKT point of the problem, that is, there exists $\bar x\in C_*$ such that
	$$\limsup_{k\to\infty}\|x^k-\bar x\|^{1/k}<1.$$
\end{theorem}  
	
	\begin{theorem}\label{Tuan_Yen's_thm} {\rm (See~\cite[Theorem~3.1]{Tuan_Yen_JOGO2013})} If~$C=\bar B(0,r)$ with $r>0,$ then for each $x^0\in C$, the DCA sequence $\{x^k\}$ obtained by the scheme~\eqref{Iteration_A} converges to a KKT point of~\eqref{QP problem}.
	\end{theorem}  
	
	 For the trust-region subproblem, Pham Dinh and Le Thi~\cite{PhamDinh_LeThi98} proved that using the scheme~\eqref{Iteration_A} with a finite number of restarts one can find a global solution of~\eqref{QP problem}. 
	
	\medskip
	Rewrite the formula in~\eqref{Iteration_A} equivalently as 	\begin{eqnarray}\label{connection} \dfrac{x^{k+1}-x^k}{\eta}=\dfrac{1}{\eta} \left[P_C\Big(x^k-\frac{1}{\rho}(Qx^k+q)\Big)-x^k\right],\end{eqnarray} where $\eta>0$ is a constant. Suppose that the whole iteration sequence $\{x^k\}$ lies on a continuously differentiable curve $\{z(t)\}_{t\geq 0}$ in $\mathbb R^n$ and $z(0)=x^0$, $z(t_k)=x^k$ for all $k$, with $\{t_k\}$ being an increasing sequence of positive real numbers satisfying the condition $\lim\limits_{k\to\infty}t_k=\infty$. Then, by using the rough approximation $\dot z(t_k)\approx \dfrac{z(t_{k+1})-z(t_k)}{\eta}=\dfrac{x^{k+1}-x^k}{\eta}$, from~\eqref{connection} we get   
	\begin{eqnarray*}\label{connection_1} \dot z(t_k)\approx\dfrac{1}{\eta} \left[P_C\Big(z(t_k)-\frac{1}{\rho}(Qz(t_k)+q)\Big)-z(t_k)\right].\end{eqnarray*} Therefore, the curve $\{z(t)\}_{t\geq 0}$ can be approximated by the trajectory $x(\cdot)$ of the dynamical system
\begin{equation}\label{dynamic_sys_A}
	\begin{cases}
		\dot{x}(t)= \dfrac{1}{\eta} \left[P_C\left(x(t)-\dfrac{1}{\rho}\big(Q x(t)+q\big)\right)-x(t)\right], \quad\mbox{\rm for all}\;\; t\geq 0,\\
		x(0)= x^0.
	\end{cases}
\end{equation}
where $\eta$ and $\rho$ are fixed positive constants and \textit{$\rho$ is strictly larger than the largest eigenvalue of $Q$.}

\medskip
If the choice~(b) is adopted, then one has an implicit iteration scheme, where $x^{k+1}$ is computed via $x^k$ by solving the strongly convex quadratic program 
\begin{eqnarray}\label{auxiliary_B}\min\Big\{\psi(x):=\frac{1}{2}x^TQx+q^Tx+\frac{\rho}{2}\|x-u\|^2\mid x\in C\Big\}\end{eqnarray}
with $u=x^k$ (see \cite{PhamDinh_LeThi_Akoa} and~\cite{CLY_2024} for more details). Denote by $F_C(u)$ the unique solution~\eqref{auxiliary_B}. Then, the map $F_C:\mathbb R^n\to C$ is Lipschitz continuous, i.e., there exists a constant $\ell>0$ such that 
\begin{equation}\label{Lipschitz_B}
\|F_C(u^1)-F_C(u^2)\|\leq\ell \|u^1-u^2\|\quad\; \forall u^1, u^2\in \mathbb R^n.
\end{equation}
This fact can be proved by the arguments for obtaining Theorem~2.1 in~\cite{Yen_1995}. Now, the implicit iteration scheme under consideration is given by the formula 
\begin{eqnarray}\label{Iteration_B} x^{k+1}=F_C(x^k)\quad  (k=0,1,2,\dots)\end{eqnarray} with $x^0\in C$ being  an initial point. The following convergence theorem was established by Cuong et al.~\cite[Theorem~3.3]{CLY_2024}.
	
\begin{theorem}\label{Cuong-Lim-Yen's_thm} If~$C$ is a polyhedral convex set and~\eqref{QP problem} has a global solution, then for each $x^0\in C$, the DCA sequence $\{x^k\}$ generated by the scheme~\eqref{Iteration_B} converges $R$-linearly to a point $\bar x\in C_*$.
\end{theorem}  

Numerical tests on two families of randomly generated indefinite quadratic programs on polyhedral convex sets in~\cite[pp.~1105--1107]{CLY_2024} showed that in terms of the number of computation steps and the execution time, algorithm~\eqref{Iteration_B}  is much more efficient than algorithm~\eqref{Iteration_A} when the algorithms
are applied to the same problem.

\medskip
If the whole sequence $\{x^k\}$ generated  by the iteration scheme~\eqref{Iteration_B} lies on a continuously differentiable curve $\{z(t)\}_{t\geq 0}$ in $\mathbb R^n$ and $z(0)=x^0$, $z(t_k)=x^k$ for all $k$, where $\{t_k\}$ is an increasing sequence of positive real numbers with $\lim\limits_{k\to\infty}t_k=\infty$, then the reasoning given in a preceding paragraph shows that the curve $\{z(t)\}_{t\geq 0}$ can be approximated by the trajectory $x(\cdot)$ of the dynamical system
\begin{equation}\label{dynamic_sys_B}
	\begin{cases}
		\dot{x}(t)= \dfrac{1}{\eta} \left[F_C(x(t))-x(t)\right], \quad\mbox{\rm for all}\;\; t\geq 0,\\
		x(0)= x^0,
	\end{cases}
\end{equation}
where $\eta$ and $\rho$ are fixed positive constants, \textit{$\rho$ is strictly larger than $-\lambda_1(Q)$}, and with $F_C(u)$ denoting the unique solution~\eqref{auxiliary_B}.

\medskip
The next questions, which are the key ingredients of solving~\eqref{QP problem} by dynamical systems, arise in a natural way. 

\medskip
\textbf{Question 1:} \textit{Let $\{t_k \}\subset (0,\infty)$ be a sequence such that $\lim\limits_{k\to\infty} t_k =\infty$ and let $\lim\limits_{k\to \infty}x(t_k )=\bar{x}$, where $x(t)$ is either the trajectory of~\eqref{dynamic_sys_A} or the trajectory of~\eqref{dynamic_sys_B}. Can we prove that $\bar{x}$ is a KKT point of~\eqref{QP problem}?}

\medskip
\textbf{Question 2:} \textit{If~\eqref{QP problem} has a solution, then there must exist a KKT point $\bar{x}$ of the problem satisfying  $\lim\limits_{t\to \infty}x(t)=\bar{x}$, where $x(t)$ is either the trajectory of~\eqref{dynamic_sys_A} or the trajectory of~\eqref{dynamic_sys_B}?}

\medskip
Regarding the positive constant $\rho$ which has a great role in guaranteeing  the convergence of the DCA sequences addressed in Theorem~\ref{Tuan's_thm} (resp., in Theorem~\ref{Cuong-Lim-Yen's_thm}), the following question is of interest.

\medskip
\textbf{Question 3:} \textit{Is the condition $\rho >\lambda_n(Q)$ (resp., $\rho >-\lambda_1(Q)$) essential for the convergence of the trajectory $x(\cdot)$ of~\eqref{dynamic_sys_A} (resp., of~\eqref{dynamic_sys_B}) to a KKT point of~\eqref{QP problem}, provided that the problem has a solution?}

\medskip
It is interesting also to investigate the role of the constant $\eta>0$ for the convergence of the trajectories of the dynamical systems~\eqref{dynamic_sys_A} and of~\eqref{dynamic_sys_B}.

\medskip
\textbf{Question 4:} \textit{Does the constant $\eta>0$ play any role in the convergence of the trajectory $x(\cdot)$ of~\eqref{dynamic_sys_A} (resp., of~\eqref{dynamic_sys_B}) to a KKT point of~\eqref{QP problem}, provided that the problem has a solution?}

\medskip
Note that using dynamical systems to solve optimization problems, variational inequalities, as well as equilibrium problems, is a powerful approach that has been studied intensively worldwide. 

\medskip
Antipin~\cite{Antipin_94} showed how one can minimize a $C^1$ function on a closed convex set by means of a dynamical system constructed by the gradient of the objective function and the metric projection onto the constraint set. Various convergence results for both first-order dynamical systems and second-order dynamical systems were obtained in~\cite[Theorems~1--5]{Antipin_94}. 

\medskip
Assuming that the variational inequality in question is \textit{strongly monotone} and a global Lipschitz property is fulfilled, Cavazzuti et al.~\cite[Theorem~5.2]{CPP_02} proved that the trajectory of a dynamical system of the type~\eqref{dynamic_sys_A} converges exponentially to the unique solution from any initial point, provided a positive coefficient $\alpha$ (see~\cite[formula~(9)]{CPP_02}) is sufficiently small. Since that $\alpha$ corresponds to the number $1/\rho$ in~\eqref{dynamic_sys_A}, the later requirement means that $\rho$ must be large enough. 

\medskip
For different methods and results of applying dynamical systems to solve optimization problems, variational inequalities, and equilibrium problems, we refer to~\cite{AGR_00,APR_14,BCPP_19,DN_93,HSV_18,Hai_2022,NZ_96,PP_02,VTV_22,Vuong_2021,VS_20}.

\medskip
Since the \textit{strong pseudomonotonicity}, a relaxed version of the strong monotonicity, of the operator defining the variational inequality in question is a crucial assumption in a number of the above-cited research works (for instance, in~\cite[Theorem~3.2]{Hai_2022}) and also in~\cite{KV_14}\cite{KVK_14}, it would be interesting to see what does the strong pseudomonotonicity means for the restriction of an affine operator on a closed convex set. Needless to say that by considering the above Questions~1--3 we hope to avoid the strong pseudomonotonicity of the restriction of the affine operator $F(x):=Qx+q$ on $C$. The hope is reasonable because the matrix $Q$ is symmetric and the number $\rho$ in~\eqref{dynamic_sys_A} and~\eqref{dynamic_sys_B} is chosen after computing the largest eigenvalue or the smallest eigenvalue of $Q$.

\medskip
The results of our preliminary investigations of the above Questions~1--4 are organized as follows. Based on the fundamental theorem of Clarke~\cite[Theorem~4.4]{Clarke_1975} on the flow-invariant sets for Lipschitzian differential inclusions, in Section~\ref{Sect-2} we establish some basic properties of the trajectories of the dynamical systems~\eqref{dynamic_sys_A} and ~\eqref{dynamic_sys_B}. Section~\ref{Sect-3} gives a complete answer to the question ``Which affine operators are strongly pseudomonotone?'' in a particular case. The question ``When the trajectories of the dynamical system converge to the Karush-Kuhn-Tucker points?'' is studied carefully in one special setting in Section~\ref{Sect-4}. 

\medskip
Although the results herein clarify many things related to Questions~1--4, a lot of work remains to be done to solve completely those questions.

\section{Trajectories of the Dynamical Systems}\label{Sect-2}

The following result on the existence and uniqueness of a global solution of the dynamical system~\eqref{dynamic_sys_A}, which is an initial value problem for differential equations (see,~e.g.,~the IVP given by formula~(2.12) in~\cite[p.~26]{Teschl_2012}), is valid.

\begin{theorem}\label{global_sol_A} For any $x^0\in\mathbb R^n$, $\rho>0$, and $\eta>0$, there exists a unique $C^1$ function $x:\mathbb R\to\mathbb R^n$ satisfying the differential equation and the initial condition in~\eqref{dynamic_sys_A}.
\end{theorem}
\begin{proof} Let $F(x)=\dfrac{1}{\eta} \left[P_C\left(x-\dfrac{1}{\rho}\big(Q x+q\big)\right)-x\right]$ for every $x\in \mathbb R^n$. Since the operator $P_C:\mathbb R^n\to C$ is nonexpansive~\cite[Chapter~I, Corollary 2.4]{ks80}, one has for any $x,y\in\mathbb R^n$ the estimates
\begin{equation*}\label{Lipschitz_property}\begin{array}{rcl}
	& & \|F(y)-F(x)\|\\
	& =&\big\|\dfrac{1}{\eta} \left[P_C\left(y-\dfrac{1}{\rho}\big(Q y+q\big)\right)-y\right]-\dfrac{1}{\eta} \left[P_C\left(x-\dfrac{1}{\rho}\big(Q x+q\big)\right)-x\right]\big\|\\
	&\leq& \dfrac{1}{\eta}\left[\big\|P_C\left(y-\dfrac{1}{\rho}\big(Q y+q\big)\right)-P_C\left(x-\dfrac{1}{\rho}\big(Q x+q\big)\right)\big\|+\|y-x\|\right]\\
	&\leq& \dfrac{1}{\eta}\left[\big\|\left(y-\dfrac{1}{\rho}\big(Q y+q\big)\right)-\left(x-\dfrac{1}{\rho}\big(Q x+q\big)\right)\big\|+\|y-x\|\right]\\
	&\leq& \dfrac{1}{\eta}\left(\dfrac{1}{\rho}\|Q\|+2\right)\|y-x\|.
\end{array}
\end{equation*} Thus, the mapping $F:\mathbb R^n\to \mathbb R^n$ is globally Lipschitz with the constant $\dfrac{1}{\eta}\left(\dfrac{1}{\rho}\|Q\|+2\right)$. Rewrite~\eqref{dynamic_sys_A} as
\begin{equation}\label{IVP}
	\dot x=F(x),\quad x(0) =x^0.
\end{equation} Applying the Picard-Lindel\"of Theorem (see, e.g., ~\cite[Chapter~2, Theorem~2.3]{Teschl_2012}), we can assert that the initial value problem~\eqref{IVP} has a locally unique solution. 

Now, to show that all solutions of the IVP in \eqref{IVP} are defined for all $t\in\mathbb R^n$ by using the result in~\cite[Chapter~2, Theorem~2.12]{Teschl_2012}, it suffices to prove that there are constants $M$ and $L$ such that
\begin{equation}\label{growth}
	\|F(x)\|\leq M+L\|x\|\quad\; \forall x\in\mathbb R^n.
\end{equation} To do so, select any point $\bar x\in C$ and note that
\begin{equation*}\label{growth_1}\begin{array}{rcl}
		\|F(x)\| & = & \big\|\dfrac{1}{\eta} \Big[P_C\left(x-\dfrac{1}{\rho}\big(Q x+q\big)\right)-x\Big]\big\|\\
		& = &  \dfrac{1}{\eta} \big\|\Big[P_C\left(x-\dfrac{1}{\rho}\big(Q x+q\big)\right)-x\Big]\big\|\\
			& = &  \dfrac{1}{\eta} \big\|\Big[\left\{P_C\left(x-\dfrac{1}{\rho}\big(Q x+q\big)\right)-\left(x-\dfrac{1}{\rho}\big(Q x+q\big)\right)\right\}-\dfrac{1}{\rho}\big(Q x+q\big)\Big]\big\|\\
			& \leq & \dfrac{1}{\eta}\Big[\big\|\bar x-\left(x-\dfrac{1}{\rho}\big(Q x+q)\right)\|+\dfrac{1}{\rho}\|Q x+q\|\Big]\\
				& \leq & \dfrac{1}{\eta}\Big[\left(\|\bar x\|+\dfrac{2}{\rho}\|q\|\right)+\left(\dfrac{2}{\rho}\|Q\|+1\right)\|x\|\Big].
	\end{array}
\end{equation*} Therefore, the growth condition~\eqref{growth} is fulfilled with
$$M:=\dfrac{1}{\eta}\left(\|\bar x-q\|+\dfrac{1}{\rho}\|q\|\right),\quad\; L:=\dfrac{1}{\eta}\left(\dfrac{2}{\rho}\|Q\|+1\right).$$

We have proved that there is a unique $C^1$ function $x(\cdot)$ defined on the whole real line satisfying the conditions in~\eqref{dynamic_sys_A}.
$\hfill\Box$
\end{proof}

\begin{theorem}\label{global_sol_B} For any $x^0\in\mathbb R^n$, $\rho>0$, and $\eta>0$, the dynamical system~\eqref{dynamic_sys_B} has a unique $C^1$ trajectory $x(\cdot)$, which is defined on the whole real line.
\end{theorem}
\begin{proof} Let $\ell$ be a positive constant satisfying the condition~\eqref{Lipschitz_B}. For every vector $x\in \mathbb R^n$, put $G(x)=\dfrac{1}{\eta} \left[F_C(x)-x\right]$ . By~\eqref{Lipschitz_B}, we have for any $x,y\in\mathbb R^n$ the following
	\begin{equation*}\label{Lipschitz_property_B}\begin{array}{rcl}
			\|G(y)-G(x)\| & =&\big\|\dfrac{1}{\eta} \big[(F_C(y)-y)-(F_C(x)-x)\big]\big\|\\
			&\leq& \dfrac{1}{\eta}\left[\big\|F_C(y)-F_C(x)\big\|+\|y-x\|\right]\\[1.5ex]
			&\leq& \dfrac{1}{\eta}(\ell+1)\|y-x\|.
		\end{array}
	\end{equation*} So, the mapping $G:\mathbb R^n\to \mathbb R^n$ is globally Lipschitz with the constant $\dfrac{1}{\eta}(\ell+1)$. Rewrite the system~\eqref{dynamic_sys_B} as
	\begin{equation}\label{IVP_B}
	\dot x=G(x),\quad x(0) =x^0.
	\end{equation} This IVP has a locally unique solution by the Picard-Lindel\"of Theorem. 
	
There are constants $M_1$ and $L_1$ such that
	\begin{equation}\label{growth_B}
	\|G(x)\|\leq M_1+L_1\|x\|\quad\; \forall x\in\mathbb R^n.
	\end{equation} Indeed, fixing a point $\bar x\in C$, we have for every $x\in \mathbb R^n$ the estimates
	\begin{equation*}\begin{array}{rcl}
		\|G(x)\| =  \big\|\dfrac{1}{\eta} \left[F_C(\bar x)-x\right]\big\| & = & \dfrac{1}{\eta}\big\|\left(F_C(x)-F_C(\bar x)\right)-\left(F_C(\bar x)-x\right)\big\|\\[1.5ex]
		& \leq &  \dfrac{1}{\eta}\Big[\ell\|x-\bar x\|+\|F_C(\bar x)-x\|\Big]\\[1.5ex]
		& \leq &  \dfrac{1}{\eta}\Big[\ell\|\bar x\|+\|F_C(\bar x)\|+(1+\ell)\|x\|\Big]  
	\end{array}
	\end{equation*} Therefore, the growth condition~\eqref{growth_B} is satisfied if we choose
	$$M_1=\dfrac{1}{\eta}\left(\ell\|\bar x\|+\|F_C(\bar x)\|\right),\quad\; L_1=\dfrac{ 1+\ell}{\eta}.$$ Thanks to~\eqref{growth_B}, by~\cite[Chapter~2, Theorem~2.12]{Teschl_2012} we can assert all solutions of the IVP in \eqref{IVP_B} are defined for all $t\in\mathbb R^n$.
		
	We have seen that the system~\eqref{dynamic_sys_B} possesses a unique $C^1$ trajectory $x(\cdot)$, which is defined on the whole real line.	$\hfill\Box$
	\end{proof}
	
	\begin{remark} {\rm In the proofs of Theorems~\ref{global_sol_A} and~\ref{global_sol_B}, there is no need to assume that $Q$ is a symmetric matrix. Hence, the obtained results are valid for general AVIs.}
	\end{remark}
	
\begin{theorem}\label{flow_invariant_A} For any $x^0\in C$, $\rho>0$, and $\eta>0$, the whole trajectory $x(\cdot)$ of~\eqref{dynamic_sys_A} is contained in $C$, that is $x(t)\in C$ for all $t\geq 0$. 
\end{theorem}
\begin{proof} Let $F(\cdot)$ be defined as in the proof of Theorem~\ref{global_sol_A}. We  have $F(\bar{x})\in T_C(\bar{x})$ for all $\bar{x}\in C$, where $T_C(\bar{x})$ denotes the Clarke tangent cone to $C$ at $\bar{x}$ (see, e.g.,~\cite[Definition~3.6]{Clarke_1975}). Indeed, for any $x\in C$,  it holds that $P_C\left(x-\dfrac{1}{\rho}\big(Q x+q\big)\right)\in C$. Hence,
\begin{equation}\label{closure_cone_1}
P_C\left(x-\dfrac{1}{\rho}\big(Q x+q\big)\right)-\bar{x}\in C-\bar{x}\subset \mbox{\rm cone} (C-\bar{x})\subset\overline{\mbox{\rm cone} (C-\bar{x})},
\end{equation}
	where $\mbox{\rm cone} (\Omega):=\{\lambda \omega\mid\lambda \geq 0,\; \omega\in \Omega\}$ is the cone generated by a subset $\Omega\subset\R^n$  and $\overline{D}$ stands for the topological closure of a subset $D\subset\R^n$. Since $C$ is convex, by  the second assertion of Proposition~3.3 and Definitions~3.6 in~\cite{Clarke_1975} (see also~\cite[Proposition~2.2.4]{Yen_2007}) we get \begin{equation}\label{tangent_cone}  T_C(\bar{x})=\overline{\mbox{\rm cone}(C-\bar{x})}.\end{equation} Combining this with~\eqref{closure_cone_1} gives
	$$P_C\left(x-\dfrac{1}{\rho}\big(Q x+q\big)\right)-\bar{x} \in T_C(\bar{x}),$$
which implies that $F(\bar{x})\in T_C(\bar{x})$. 
	
Moreover, as $F$ is Lipschitz continuous by the proof of Theorem~\ref{global_sol_A}, it follows from~\cite[Theorem~4.4]{Clarke_1975} that $C$ is flow-invariant for $F$. Thus, for any trajectory $x(\cdot)$ with $x^0\in C$ of~\eqref{dynamic_sys_A}, we have $x(t)\in C$ for all $t\geq 0$.
		
	$\hfill\Box$
	\end{proof}
	
	\begin{theorem}\label{flow_invariant_B} For any $x^0\in C$, $\rho>0$, and $\eta>0$, the whole trajectory $x(\cdot)$ of~\eqref{dynamic_sys_B} is contained in $C$. 
	\end{theorem}
	\begin{proof} Let $G(\cdot)$ be defined as in the proof of Theorem~\ref{global_sol_B}. Take any $\bar{x}\in C$. Since $F_C(\bar{x})\in C$, we obtain
	\begin{equation}\label{closure_cone_2}
	F_C(\bar{x})-\bar{x}\in C-\bar{x}\subset \mbox{\rm cone} (C-\bar{x})\subset\overline{\mbox{\rm cone} (C-\bar{x})}.
	\end{equation}
	As $C$ is a convex set, the Clarke tangent cone $T_C(\bar{x})$ to $C$ at $\bar{x}$ can be computed by~\eqref{tangent_cone}. This together with the inclusions in~\eqref{closure_cone_2} implies that $G(\bar{x})\in T_C(\bar{x})$. Furthermore, as shown in the proof of Theorem~\ref{global_sol_B}, $G$ is Lipschitz continuous. Since $\bar{x}\in C$ was chosen arbitrarily, we deduce from~\cite[Theorem~4.4]{Clarke_1975} that  $C$ is flow-invariant for $G$. Consequently, $x(t)\in C$ for all $t\geq 0$.
		$\hfill\Box$
	\end{proof}

	\begin{remark} {\rm 
	Since the proofs of Theorems~\ref{flow_invariant_A} and~\ref{flow_invariant_B}, do not rely on the assumption that $Q$ is a symmetric matrix, both theorems are applicable to general AVIs.}
	\end{remark}

\section{Which Affine Operators are Strongly Pseudomonotone?}\label{Sect-3}

As in~\cite[Definition~2.3.1]{FP_2003} and~\cite[Definition~2.1]{KV_14}, an operator $F:C\to\mathbb R^n$, where $C$ is a convex subset of $\mathbb R^n$, is said to be
	
(a) \textit{monotone} if $\langle F(y)-F(x),y-x\rangle\geq 0$ for all $x,y\in C$;

(b) \textit{pseudomonotone} if for all $x,y\in C$ it holds
$$ \langle F(x),y-x\rangle\geq 0\ \; \Longrightarrow\ \; \langle F(y),y-x\rangle\geq 0;$$
\hskip0.5cm (c) \textit{strongly monotone} if there exists a constant $\gamma>0$ such that $$\langle F(y)-F(x),y-x\rangle\geq\gamma\|y-x\|^2$$ for all $x,y\in C$;

(d) \textit{strongly pseudomonotone} if there exists a constant $\gamma>0$ such that 
\begin{equation}\label{gamma_spm}
\langle F(x),y-x\rangle\geq 0\ \; \Longrightarrow\ \; \langle F(y),y-x\rangle\geq \gamma\|y-x\|^2
\end{equation}
for all $x,y\in C$.

\medskip
If the property~\eqref{gamma_spm} holds for all $x,y\in C$, then one says that $F$ is\textit{~$\gamma$-strongly pseudomonotone on $C$}.

\begin{remark}\label{conic_property} {\rm If $F:C\to\mathbb R^n$ is a monotone (resp., pseudomonotone, strongly monotone, strongly pseudomonotone) operator and $\lambda$ is a positive real number, then the operator  $\lambda F:C\to\mathbb R^n$ defined by setting $(\lambda F)(x)=\lambda F(x)$ for all $x\in C$ is also monotone (resp., pseudomonotone, strongly monotone, strongly pseudomonotone). Thus, the set of  monotone (resp., pseudomonotone, strongly monotone, strongly pseudomonotone) operators from $C$ to $\R^n$ is a cone.}
\end{remark}

\medskip
As it has been noted in Section~\ref{Sect-1}, strong pseudomonotonicity of the operator defining a variational inequality is an important property, and many existing results rely on the latter. In an infinite-dimensional Hilbert space setting, Khanh and Vuong~\cite[Example~4.1]{KV_14} gave some examples of strongly pseudomonotone operators, which are neither strongly monotone nor monotone. Given by the formula $F_\beta(x)=(\beta-\|x\|)x$ on the set $C_r:=\{x\mid \|x\| \leq r\}$, where the constants $\beta$ and $r$ satisfy the conditions $\beta>r>\dfrac{\beta}{2}>0$, these operators are non-affine. In \cite{KV_14}, it is proved that $F_\beta$ is strongly pseudomonotone on $C_r$ with the coefficient $\alpha:=\beta-r>0$.

\medskip

Focusing on the restrictions of affine operators on convex sets (called \textit{affine operators on convex sets} for brevity), we have the following natural question. 

\medskip
\textbf{Question~5:} \textit{Whether there is a strongly pseudomonotone affine operator on a convex set, which is not strongly monotone on that set?}

\medskip The next example gives an answer in the affirmative to Question~5.

\begin{example}\label{1st_example}
{\rm Let $A: \R\to\R$ be defined by $A(x)=-x+1$ for $x\in\R$ and let $C=\left[0,\frac{1}{2}\right]$. For any $x,y\in C$, if 
\begin{equation}\label{ineq1}
		(-x+1)(y-x)\geq 0,
	\end{equation} 
then $y-x\geq 0$. Since $x,y\in C$, we also have $1\geq 2y-x$ or, equivalently, $1-y\geq y-x$. As $y-x\geq 0$, we can multiply both sides of the last inequality by $(y-x)$ to get
\begin{equation}\label{ineq2}
(-y+1)(y-x)\geq (y-x)^2.
\end{equation} 
Therefore, we have shown that if~\eqref{ineq1} holds, i.e., $\la A(x),y-x\ra \geq 0,$ then~\eqref{ineq2} is valid, i.e., $\la A(y),y-x\ra \geq |y-x|^2.$
So, $A$ is $\gamma$-strongly pseudomonotone on~$C$ with the constant $\gamma:=1$.

The affine function $A$ is non-monotone on $C$. Indeed, for any $x,y\in C$ with $x\neq y$ we have
\begin{equation*}
	\la A(y)-A(x),y-x\ra = \big[ (-y+1)-(-x+1)\big](y-x)=-(y-x)^2<0.
\end{equation*}
\hskip 0.5cm Thus, we have seen that the affine function $A$ is strongly pseudomonotone on~$C$, but not monotone on $C$.
}
\end{example}

Example~\ref{1st_example} shows that for affine operators on convex sets, strong pseudomonotonicity is a genuine relaxed version of strong monotonicity. In addition, the example demonstrates that strongly pseudomonotone affine operators on convex sets on convex sets may be non-monotone as well.

\medskip
The subset $C=\left[0,\frac{1}{2}\right]$ of $\mathbb R$ in Example~\ref{1st_example} is non-symmetric with respect to 0. So, it is of interest to examine the effects produced by the example for affine operators on convex sets which are symmetric w.r.t. the origin of the space under consideration. Note that the closed unit ball in $\mathbb R$ is exactly the segment $[-1,1]$, which is also a compact polyhedral convex set with nonempty interior. Hence, the next question is of interest as it seeks for a \textit{complete characterization of strong monotonicity for affine operators defined on the closed unit ball} of a special Euclidean space.

\medskip
\textbf{Question 6:} \textit{Let $\alpha,\beta\in\R$ and $F:\R\to \R$ be defined by $F(x)=\alpha x + \beta$ for $x\in\R$, where  $\alpha,\beta\in\R$. Under which conditions on $(\alpha,\beta)$ is $F$ strongly pseudomonotone on $C:=[-1,1]$?}

\begin{theorem}\label{Thm3.1}
	Let $F:\R\to \R$ be defined by $F(x)=\alpha x + \beta$ for $x\in\R$ with $\alpha,\beta\in\R$, and $C:=[-1,1]$. Then, $F$ is strongly pseudomonotone if and only if $(\alpha,\beta)$ belongs to the cone
	\begin{equation*}
		\mathcal{A}:= \big\{(\alpha,\beta)\in  \R^2\mid \alpha + \beta >0 \;\;\text{\rm or}\;\; \alpha-\beta  > 0\big\}.
	\end{equation*}
\end{theorem}

\begin{figure}[!ht]
	\begin{center}
		\includegraphics[width=0.5\textwidth]{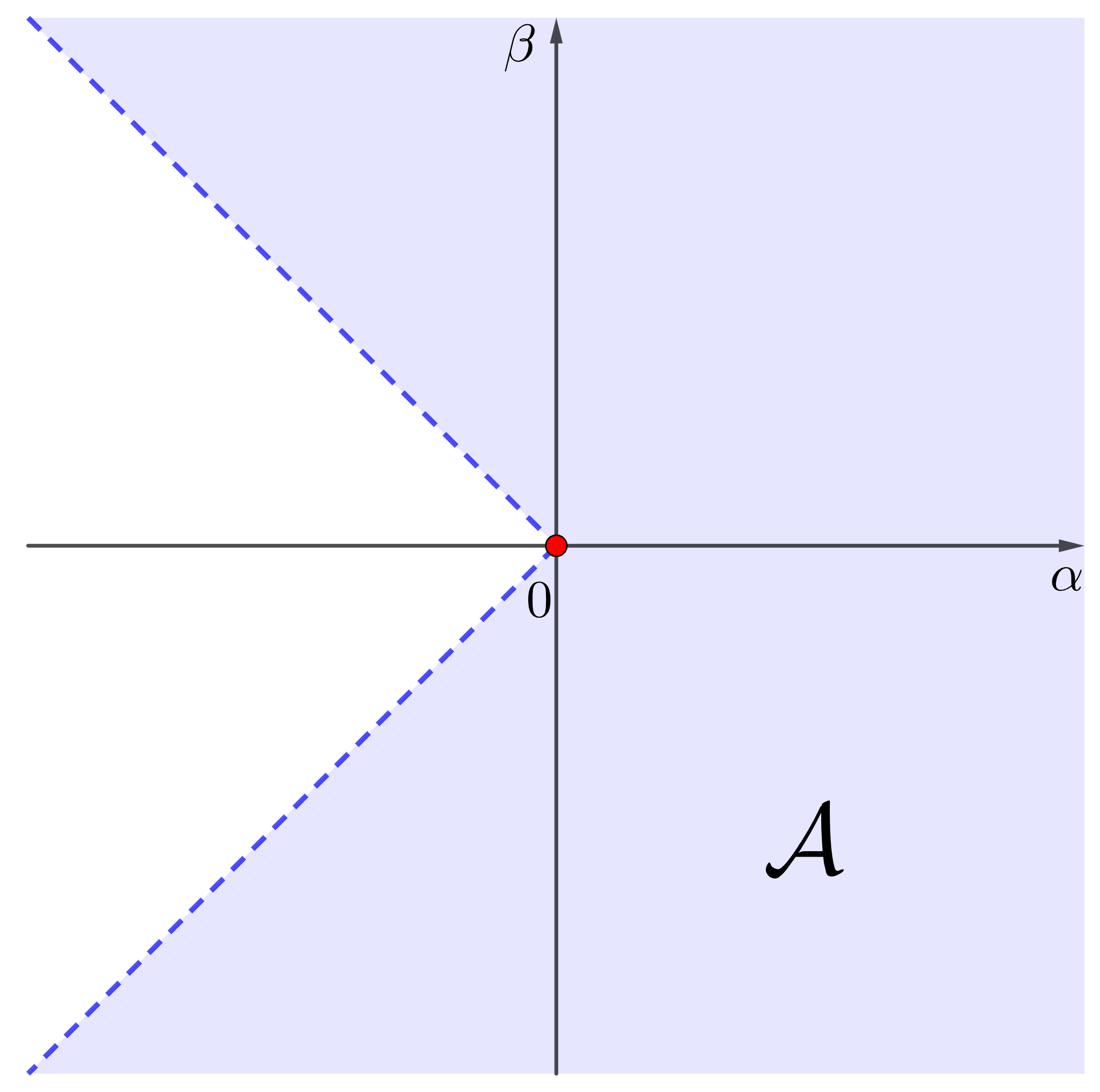}
		\caption{The cone $\mathcal{A}$ is non-closed}
	\end{center}
\end{figure}
\begin{proof}  Depending on the position of the parameter pair $(\alpha,\beta)$ in $\R^2$, we have the following six cases.
	
{\bf Case 1:} $\alpha =\beta=0$. We have $F(x)\equiv 0$ for all $x\in \R$. Thus, $F$ is not strongly pseudomonotone on $C$.
	
	{\bf Case 2:} $\alpha = 0$ and $\beta \neq 0$. In this case, $F(x)=\beta$ for all $x\in \R$. If $x,y \in C$ are such that $\beta(y-x) \geq 0$, then we have
	\begin{equation*}
		\beta(y-x) = |\beta||y-x| =\dfrac{|\beta|}{2} (2|y-x|) \geq \dfrac{|\beta|}{2} (y-x)^2,
	\end{equation*}
	where the last inequality follows from the fact that $|y-x|\leq 2$ for any $x,y\in C$. So, $F$ is~$\gamma$-strongly pseudomonotone on $C$ with $\gamma =\dfrac{|\beta|}{2}$.
	
    {\bf Case 3:} $\alpha >0$. Then, for all $x,y\in C$ and for all $\beta\in\R$, one has
	\begin{equation*}
		\langle F(y)- F(x), y-x\rangle = \big[(\alpha y +\beta) -(\alpha x +\beta)\big](y-x)= \alpha (y-x)^2.
	\end{equation*}
	Therefore, $F$ is $\gamma$-strongly monotone on $C$ for $\gamma:=\alpha$.
	
	{\bf Case 4:} $\alpha <0$ and $\beta=0$. Letting $x=-1$ and $y=1$, we have
	$$\la F(x),y-x\ra = \alpha x (y-x) = -2\alpha >0,$$
	while 
	$$\la F(y),y-x\ra=\alpha y (y-x)= 2\alpha <0.$$
	Hence, $F$ is not pseudomonotone on $C$.
	
	{\bf Case 5:} $\alpha <0$ and $\beta >0$. We will consider the 3 following subcases.
	
	$\quad${\bf Subcase 5a:} $\alpha< - \beta$, i.e., $\alpha + \beta < 0$. We have  $ \dfrac{-\beta}{\alpha} \in (0,1)$. Choose $x= \dfrac{-\beta}{\alpha}$ and $y=1$. Then, on one hand it holds that 
	$$\la F(x),y-x\ra =(\alpha x+\beta)(y-x)= \left(\alpha\dfrac{-\beta}{\alpha}+\beta\right)\left(1-\dfrac{-\beta}{\alpha}\right)=0.$$
	On the other hand, we have
	$$\la F(y),y-x\ra = (\alpha y+\beta)(y-x) = (\alpha + \beta)\left(1- \dfrac{-\beta}{\alpha}\right) < 0.$$
	Therefore, $F$ is not pseudomonotone on $C$.
	
	$\quad${\bf Subcase 5b:} $\alpha= -\beta$, i.e., $\alpha+\beta=0$. Taking $x=1$ and $y=-1$, we have
	$$\la F(x),y-x\ra = (\alpha x+\beta) (y-x) =-2(\alpha +\beta)=0,$$
	and 
	$$\la F(y),y-x\ra=(\alpha y +\beta) (y-x)= -2(-\alpha +\beta) =-2\beta <0.$$
	This implies that $F$ is not pseudomonotone on $C$.
	
	$\quad${\bf Subcase 5c:} $\alpha> -\beta$, i.e., $\alpha + \beta > 0$. Since $x\leq 1$ for all $x\in C$, it follows that
	$$-\alpha x \leq -\alpha <\beta,$$
	which implies $\alpha x +\beta >0$ for all $x\in C$. Let $x,y\in C$ be such that 
	$$\la F(x),y-x\ra = (\alpha x+\beta) (y-x) \geq 0.$$
	This condition is equivalent to $y-x \geq 0$. As $\alpha <0$ and $x\leq 1$, we have
	$$\alpha x+\beta \geq \alpha +\beta = \dfrac{\alpha+\beta}{2}-\dfrac{\alpha+\beta}{2}(-1).$$
	Moreover, since $\alpha + \beta > 0$, $y\leq 1$ and $x \geq -1$, this yields
	$$\alpha x+\beta\geq  \dfrac{\alpha+\beta}{2}y-\dfrac{\alpha+\beta}{2}x.$$
	Multiplying both sides of the last inequality by $(y-x)$, we obtain
	$$\la F(y),y-x\ra= (\alpha y +\beta) (y-x)\geq \dfrac{\alpha+\beta}{2}(y-x)^2.$$
	Thus, $F$ is $\gamma$-strongly pseudomonotone on $C$ with $\gamma =\dfrac{\alpha+\beta}{2}$.
	
	{\bf Case 6:} $\alpha <0$ and $\beta <0$. In this case,  $\dfrac{-\beta}{\alpha}<0$.
	
	$\quad${\bf Subcase 6a:} $\alpha \leq \beta$. Then, $\dfrac{\beta}{\alpha} \leq 1$ or, equivalently, $-1\leq \dfrac{-\beta}{\alpha}$. Choose $x= \dfrac{-\beta}{\alpha}$, $y=1$, and observe that $x,y\in C$. Clearly,
	$$\la F(x),y-x\ra = (\alpha x+\beta) (y-x) = 0.$$
	Since $\alpha +\beta <0$ and $\dfrac{-\beta}{\alpha}<0$, we have
	$$\la F(y),y-x\ra= (\alpha y +\beta) (y-x) = (\alpha +\beta) \left(1-\dfrac{-\beta}{\alpha}\right) <0.$$
	Hence, $F$ is not pseudomonotone on $C$.
	
	$\quad${\bf Subcase 6b:} $\alpha > \beta$. This implies that  $\alpha-\beta>0$ and $\dfrac{-\beta}{\alpha} <-1$. For any $x\in C$, one has $x\geq -1 >\dfrac{-\beta}{\alpha}$. Thus, $\alpha x + \beta <0$ for all $x\in C$. Let $x,y\in C$ be such that
	$$\la F(x),y-x\ra = (\alpha x+\beta) (y-x) \geq 0.$$
	Since $\alpha x + \beta <0$ for every $x\in C$, we have $y-x \leq 0$. In addition, as $y\geq -1$ and $\alpha <0$, we have
	$$\alpha y +\beta \leq -\alpha +\beta= -\dfrac{\alpha-\beta}{2}-\dfrac{\alpha-\beta}{2} = \dfrac{\alpha-\beta}{2}(-1) - \dfrac{\alpha-\beta}{2}.$$
	Combining this with the fact that $x,y\in [-1,1]$ gives
	$$\alpha y +\beta \leq \dfrac{\alpha-\beta}{2} y -\dfrac{\alpha-\beta}{2} x.$$
	Multiplying both sides of the last inequality by $(y-x)$ yields
	$$\la F(y),y-x\ra= (\alpha y +\beta) (y-x)  \geq  \dfrac{\alpha-\beta}{2}(y-x)^2.$$
	Therefore, $F$ is $\gamma$-strongly pseudomonotone on $C$ with $\gamma =  \dfrac{\alpha-\beta}{2}$.
	
	In conclusion, $F$ is strongly pseudomonotone on $C$ if and only if $(\alpha,\beta)$ is such that one of the Cases~2 and 3, or one of the Subcases 5c and 6b, occurs. Equivalently, $F$ is strongly pseudomonotone on $C$ if and only if  $(\alpha,\beta)\in\mathcal{A}$. $\hfill\Box$
\end{proof}

\begin{remark} The cone  $\mathcal{A}$ is nonconvex. Indeed, taking any $(\alpha_1,\beta_1)$ and $(\alpha_2,\beta_2)$ in $\R^2$ with $\alpha_1<0$, $\alpha_2<0$, $\beta_1>-\alpha_1$ and $\beta_2<\alpha_2$, we see that the line segment joining $(\alpha_1,\beta_1)$ with $(\alpha_2,\beta_2)$ does not lie entirely in $\mathcal{A}$. If one requires in addition that $\beta_1+\beta_2=0$, then the midpoint of that segment does not belong to $\mathcal{A}$. Since the latter is a cone, $(\alpha_1,\beta_1)+(\alpha_2,\beta_2)\notin \mathcal{A}$. This means that \textit{the sum of two strongly pseudomonotone affine operators on a compact polyhedral convex set in $\mathbb R^n$ may not be strongly pseudomonotone on the set.} To have another example of this type, take any $(\alpha,\beta)\in \mathcal{A}$ with $\alpha<0$ and note that $(-\alpha,-\beta)\in\mathcal{A}$, but the sum  $(\alpha,\beta)+(-\alpha,-\beta)=(0,0)$ does not belong to $\mathcal{A}$.
\begin{figure}[!ht]
	\centering
	\includegraphics[width=0.5\textwidth]{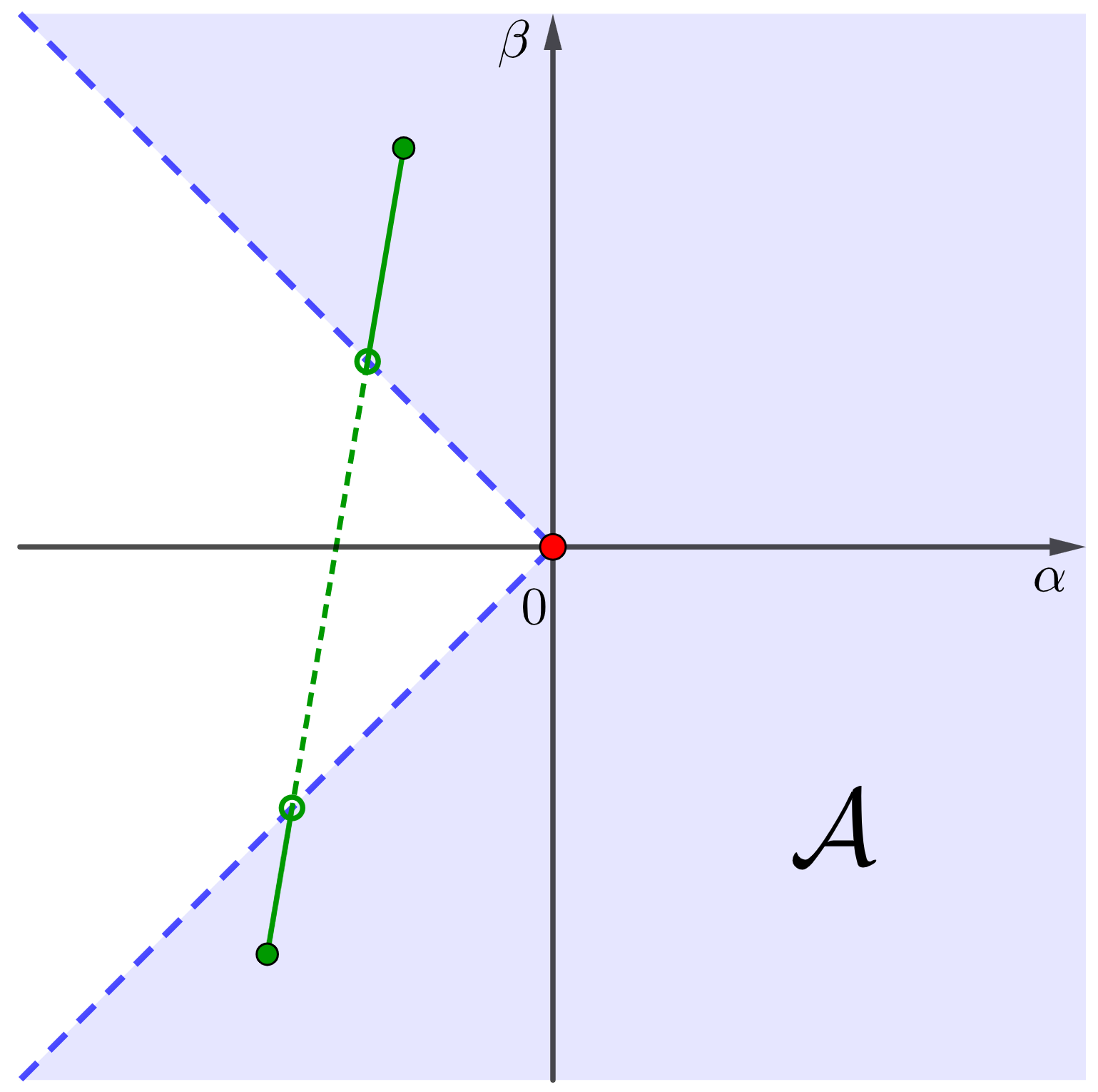}
	\caption{The cone $\mathcal{A}$ is nonconvex}
\end{figure}
\end{remark}
\section{When the Trajectories of the Dynamical System Converge to the Karush-Kuhn-Tucker Points?}\label{Sect-4}

As in Theorem~\ref{Thm3.1}, let $F:\R\to \R$ be defined by $F(x)=\alpha x + \beta$ for $x\in\R$ with $\alpha,\beta$ being real constants, and let $C=[-1,1]$. Consider the (possibly nonconvex) quadratic optimization problem
\begin{equation}\label{trusted_reg_sub_1}
	\min \left\{f(x):=\dfrac{1}{2} \alpha x^2 +\beta x \mid x\in C\right\}.
\end{equation}
According to the generalized Fermat's rule (see, for example,~\cite[p.~2]{ks80} and~\cite[Theorem~3.1]{lty05}), if $\bar{x}\in C$ is a local solution of~\eqref{trusted_reg_sub_1}, then 
\begin{equation*}
	\nabla f(\bar{x}) (y-\bar{x}) \geq 0,\quad\mbox{\rm for all}\;\;y\in C,
\end{equation*}
or, equivalently, 
\begin{equation}\label{fermat_rule_1}
	(\alpha\bar{x}+\beta) (y-\bar{x}) \geq 0,\quad\mbox{\rm for all}\;\;y\in C.
\end{equation}
In agreement with the definition recalled in Section~\ref{Sect-1}, $\bar{x}\in C$ is said to be a Karush-Kuhn-Tucker point (a KKT point) of~\eqref{trusted_reg_sub_1} if the condition~\eqref{fermat_rule_1} is satisfied.

For some $x^0\in C$, let $x(\cdot)$ be the solution of the following dynamical system
\begin{equation}\label{dynamic_sys_A1}
	\begin{cases}
		\dot{x}(t)= \dfrac{1}{\eta} \left[P_C\left(x(t)-\dfrac{1}{\rho}\big(\alpha x(t)+\beta\big)\right)-x(t)\right], \quad\mbox{\rm for all}\;\; t>0,\\
		x(0)= x^0.
	\end{cases}
\end{equation}
where $\eta$ and $\rho$ are fixed positive constants.

\medskip
\textbf{Question 7:} \textit{Does $x(t)$ converge to a KKT point of~\eqref{trusted_reg_sub_1} as $t\to\infty$ for any $(\alpha,\beta)\in\R\times\R$  and for any $x^0\in [-1,1]$, or not?}

\medskip
A partial answer to this question is given in the next proposition.

\begin{proposition}\label{prop1:dim1}
	For any $(\alpha,\beta)\in\R\times\R$ satisfying $\alpha<0$ and for any $x^0\in [-1,1]$, the unique solution $x(t)$ of~\eqref{dynamic_sys_A1}  converges to a KKT point of~\eqref{trusted_reg_sub_1} as $t\to\infty$. 
\end{proposition}
\begin{proof}
	Let $\alpha<0$ and $x^0\in [-1,1]$ be given arbitrarily. Let $x(t)$ be the unique solution of the initial value problem~\eqref{dynamic_sys_A1}. To study the behavior of $x(t)$,   note that  
		\begin{equation}\label{proj_u} P_C(u)=\begin{cases}
		-1\qquad &\mbox{\rm for}\;\;u<-1,\\
		1&\mbox{\rm for}\;\;u>1,\\
		u&\mbox{\rm for}\;\;-1\leq u\leq1.
	\end{cases}
	\end{equation}
	Furthermore, for our convenience,  we put $\mu_1= \dfrac{\beta-\rho}{\rho-\alpha}$, $\mu_2 =\dfrac{\beta+\rho}{\rho-\alpha}$, and observe from the inequalities  $\rho -\alpha >0$ and $\rho >0$  that $\mu_1<\mu_2$. For every $x\in \R$, let $u(x)=x-\dfrac{1}{\rho}(\alpha x+\beta)$. Then,  the following assertions hold:
	\begin{itemize}
		\item[\rm (i)] $u(x)<-1$ if and only if $x<\mu_1$;
		\item[\rm (ii)] $u(x)>1$ if and only if $x>\mu_2$;
		\item[\rm (iii)]  $-1\leq u(x) \leq 1$  if and only if $\mu_1\leq x\leq \mu_2$.
	\end{itemize}
	Therefore, the differential relation in~\eqref{dynamic_sys_A1} means that
	\begin{equation}\label{dot_xt}
		\dot{x}(t)=\begin{cases}
			\dfrac{1}{\eta}\big(-1-x(t)\big)\qquad &\mbox{\rm if}\;\; x(t) < \mu_1,\\[2ex]
			\dfrac{1}{\eta}\big(1-x(t)\big) &\mbox{\rm if}\;\; x(t) > \mu_2,\\[2ex]
			-\dfrac{1}{\eta\rho}\big(\alpha x(t) +\beta\big)&\mbox{\rm if}\;\;\mu_1\leq x(t)\leq \mu_2.
		\end{cases}
	\end{equation}
	Based on the relationships between the initial point $x^0$ and the segment $[\mu_1,\mu_2]\subset \R$,  we distinguish the next cases.
	
	\smallskip
	{\bf Case 1:} $x^0 < \mu_1$. In this case,  we have $-1\leq x^0< \mu_1=\dfrac{\beta-\rho}{\rho-\alpha}$. It follows that $\alpha-\rho<\beta-\rho$; hence $\beta > \alpha$. So, $\bar{x} =-1$ satisfies the condition~\eqref{fermat_rule_1}, i.e., $\bar{x}=-1$ is a KKT point of~\eqref{trusted_reg_sub_1}. Since $x^0<\mu_1$, by the continuity of $x(\cdot)$ we can find some $\tau >0$ such that 
	$$x(t) < \mu_1\quad\mbox{for all} \;\;t\in [0,\tau].$$
	Hence, using~\eqref{dot_xt}, we have the following initial value problem for a first-order linear ordinary differential equation with constant coefficients
	\begin{equation*}\label{dynamic_sys_A1_c1}
		\begin{cases}
			\dot{x}(t)= -\dfrac{1}{\eta}x(t)-\dfrac{1}{\eta}, \quad\mbox{\rm for all}\;\; t\in[0,\tau],\\
			x(0)= x^0.
		\end{cases}
	\end{equation*}
	So, from the result in~\cite[Formula~(3.43), p.~53]{Teschl_2012} it follows that
	\begin{equation}\label{ode_sol_1}
		x(t)=\left(x^0+1\right)e^{-\frac{1}{\eta}t}-1
	\end{equation}
	for all $t\in[0,\tau]$. Since $e^{-\frac{1}{\eta}t} >0$ for all $t$ and $x^0 <\mu_1$, one has
	$$x(t)=\left(x^0+1\right)e^{-\frac{1}{\eta}t}-1 < \left(\mu_1+1\right)e^{-\frac{1}{\eta}t}-1,$$
	for all $t\in [0,\tau]$.
	Moreover, as $\mu_1>-1$, this implies that
	$$x(t) < \left(\mu_1+1\right)-1 = \mu_1, \quad\mbox{\rm for all}\;\; t\in [0,\tau].$$
	In particular, $x(\tau)<\mu_1$. 
	
	\smallskip
	{\sc Claim 1:} \textit{The solution~\eqref{ode_sol_1} can be extended to the whole half-line $[0,\infty)$.}
	
	\smallskip
	Indeed, letting $\tau$ play the role of the previous initial time $t=0$ and $x(\tau)$ play the role of the former $x^0$, we can deduce from the previous passage of proof that $x(t)$ is of the form~\eqref{ode_sol_1} for all $t\in[\tau,2\tau]$. Then, letting $2\tau$ play the role of the initial time and $x(2\tau)$ play the role of the initial value, we see that $x(t)$ is of the form~\eqref{ode_sol_1} for all $t\in[2\tau,3\tau]$. Therefore, by induction we can extend the formula~\eqref{ode_sol_1} for $x(t)$ to $[0,\infty)$.
	
	According to Claim 1,  we have $x(t)= \left(x^0+1\right)e^{-\frac{1}{\eta}t}-1$ for all $t\geq 0$.  It follows that $\lim\limits_{t\to\infty} x(t) =-1$. In other words, under the assumption made in this case,  $x(t)$ converges to the KKT point $\bar{x}=-1$ of~\eqref{trusted_reg_sub_1} as $t\to\infty$. 
	
	\smallskip
	{\bf Case 2:} $x^0 =\mu_1$. Since $-1\leq x^0=\mu_1=\dfrac{\beta-\rho}{\rho-\alpha}$, one has $\alpha \leq \beta$. Therefore,
	\begin{equation}\label{eq1}
		\mu_1+\dfrac{\beta}{\alpha} =\dfrac{\beta-\rho}{\rho-\alpha} +\dfrac{\beta}{\alpha} =\dfrac{\rho(\beta-\alpha)}{\alpha(\rho-\alpha)}\leq 0.
	\end{equation}
	Setting $\bar{x}=-1$, we deduce from the relation $\beta-\alpha\geq 0$ and~\eqref{fermat_rule_1} that $\bar{x}$ is a KKT point of~\eqref{trusted_reg_sub_1}. 
	
	If $\alpha = \beta$, then we have $x^0= \dfrac{\beta-\rho}{\rho-\alpha}=-1$.  As $x(t)\in C$ for all $t\geq 0$, it holds that
	$$x(t) \geq-1=\mu_1\quad \mbox{\rm for all}\;\; t\geq 0.$$  Since $\mu_1<\mu_2$, by the continuity of  $x(\cdot)$ we can find a number $\tau >0$ such that $\mu_1\leq x(t)\leq \mu_2$ for all $t\in[0,\tau]$.  Combining this with~\eqref{dot_xt} yields
	\begin{equation*}\label{dynamic_sys_A1_c2}
		\begin{cases}
			\dot{x}(t)= 	-\dfrac{\alpha}{\eta\rho} x(t) -\dfrac{\beta}{\eta\rho} , \quad\mbox{\rm for all}\;\; t\in[0,\tau],\\
			x(0)= x^0.
		\end{cases}
	\end{equation*}
	Hence, by~\cite[Formula~(3.43), p.~53]{Teschl_2012} we get
	\begin{equation}\label{ode_sol_2}
		x(t)=	\left(x^0+\dfrac{\beta}{\alpha}\right)e^{\frac{-\alpha}{\eta\rho}t} -\dfrac{\beta}{\alpha},
	\end{equation}
	for all $t\in[0,\tau]$.
	Recalling that $\alpha=\beta$ and $x^0=-1$, from this we can infer that $x(t)=-1$ for all $t\in [0,\tau]$. In particular, $x(\tau)=-1$. Arguing similarly as in the proof of Claim 1, we get $x(t)=-1$ for every $t\in [0,\infty)$. So, the assertion of the proposition is valid.
	
	Now, assume that $\alpha <\beta$. We have  $x(0)=x^0=\mu_1$ and $\mu_1<\mu_2$. If there exists $\bar{t} >0$ such that $x(\bar{t}) < \mu_1$ then, arguing similarly as in Case 1, we can show that $$x(t)=\left(x(\bar{t})+1\right)e^{-\frac{1}{\eta}(t-\bar{t})}-1$$ for every $t\in [\bar{t},\infty)$. It follows that $\lim\limits_{t\to \infty} x(t) =-1$, i.e.,  $x(t)$ tends to the KKT point $\bar{x}=-1$ of~\eqref{trusted_reg_sub_1} as $t\to\infty$.  If $x(t)\geq \mu_1$ for all $t\geq 0$ then, by the continuity of  $x(\cdot)$, there exists $\tau >0$ such that $\mu_1\leq x(t)\leq \mu_2$ for all $t\in[0,\tau]$. So, applying~\eqref{dot_xt} and~\cite[Formula~(3.43), p.~53]{Teschl_2012}, we can infer that  $x(t)$ is of the form~\eqref{ode_sol_2} for all $t\in[0,\tau]$. Furthermore, since $\alpha <\beta$, the last inequality in~\eqref{eq1} is strict; hence, $\mu_1+ \dfrac{\beta}{\alpha} <0$. Besides, as $\alpha <0$, we have $e^{\frac{-\alpha}{\eta\rho}t} > 1$ for all $t> 0$. Therefore, for every $t\in(0,\tau]$, it holds that 
	$$x(t)=\left(x^0+\dfrac{\beta}{\alpha}\right)e^{\frac{-\alpha}{\eta\rho}t} -\dfrac{\beta}{\alpha}= \left(\mu_1+\dfrac{\beta}{\alpha}\right)e^{\frac{-\alpha}{\eta\rho}t} -\dfrac{\beta}{\alpha} <	\left(\mu_1 +\dfrac{\beta}{\alpha}\right) -\dfrac{\beta}{\alpha} = \mu_1.$$
	We have arrived at a contradiction. Thus, the situation $x(t) \geq \mu_1$ for all $t\geq 0$ is excluded.

	{\bf Case 3:} $x^0>\mu_2$. Since $x^0\leq 1$ and $\mu_2=\dfrac{\beta+\rho}{\rho-\alpha}$, this implies that $\beta <-\alpha$. For $\bar{x}:=1$, we have
	$$\alpha \bar{x}+\beta =\alpha+\beta <0.$$
	So, from~\eqref{fermat_rule_1} it follows that $\bar{x}=1$ is a KKT point of~\eqref{trusted_reg_sub_1}. As $x(\cdot)$ is continuous on $[0,\infty)$ and $x(0)=x^0>\mu_2$, there is $\tau>0$ such that $x(t)> \mu_2$
	for all $t\in[0,\tau]$. Therefore, by~\eqref{dot_xt} we have
	\begin{equation*}\label{dynamic_sys_A1_c4}
		\begin{cases}
			\dot{x}(t)= -\dfrac{1}{\eta}x(t)+\dfrac{1}{\eta}, \quad\mbox{\rm for all}\;\; t\in[0,\tau]\\
			x(0)=x^0.
		\end{cases}
	\end{equation*}
	Using~\cite[Formula~(3.43), p.~53]{Teschl_2012} gives
	\begin{equation*}\label{ode_sol_3} 
		x(t)= \left(x^0-1\right)e^{-\frac{1}{\eta}t}+1,
	\end{equation*}
	for all $t\in[0,\tau]$. Hence, 
	$$x(t)= \left(x^0-1\right)e^{-\frac{1}{\eta}t}+1> \left(\mu_2 -1\right)e^{-\frac{1}{\eta}t}+1,$$
	for all $t\in[0,\tau]$. Since $e^{-\frac{1}{\eta}t}\leq 1$ for every $t\geq 0$ and $$\mu_2-1 < x^0-1\leq 0,$$ this strict lower estimate for $x(t)$ yields
	$$x(t) > \left(\mu_2 -1\right)+1 = \mu_2,$$
	for all $t\in[0,\tau]$. In particular, $x(\tau) >\mu_2$. Arguing similarly as in the proof of Claim 1, we have $x(t)= \left(x^0-1\right)e^{-\frac{1}{\eta}t}+1$ for all $t\geq 0$. So, $x(t)$ tends to the KKT point $\bar{x}=1$ as $t\to\infty$.

	{\bf Case 4:} $x^0=\mu_2$. Then, we have $$1\geq x^0=\mu_2= \dfrac{\beta+\rho}{\rho-\alpha}.$$ It follows that $\alpha+\beta \leq 0$. 
	For $\bar{x}:=1$, since
	$$\alpha \bar{x}+\beta =\alpha+\beta \leq 0,$$
	by~\eqref{fermat_rule_1} we can assert that $\bar{x}=1$ is a KKT point of~\eqref{trusted_reg_sub_1}. As $\alpha+\beta \leq 0$, we see that either  $\alpha = -\beta$, or $\alpha < -\beta$.
	
	If $\alpha = -\beta$, then $x^0=\mu_2=\dfrac{\beta+\rho}{\rho -\alpha}=1$. Since $x(t)\in C=[-1,1]$ for all $t\geq 0$, we have $x(t)\leq 1=\mu_2$ for all $t\geq 0$. Therefore, by the continuity of $x(\cdot)$, we can find some $\tau >0$ such that $\mu_1 < x(t)\leq \mu_2$
	for all $t\in[0,\tau]$. Combining this with~\eqref{dot_xt} and~\cite[Formula~(3.43), p.~53]{Teschl_2012}, one gets the expression~\eqref{ode_sol_2} for all $t \in [0, \tau]$. Since $\alpha = -\beta$ and $x^0=\mu_2$, from~\eqref{ode_sol_2} it follows that
	$$x(t)= \left(\dfrac{\beta+\rho}{\rho-\alpha}+\dfrac{\beta}{\alpha}\right)e^{\frac{-\alpha}{\eta\rho}t} -\dfrac{\beta}{\alpha} =1,$$
	for all $t\in[0,\tau]$. Then, arguing similarly as in the proof of Claim 1, we obtain $x(t) =1$ for all $t\geq 0$. So, $x(t)$ converges to the KKT point $\bar{x}=1$ as $t\to\infty$.
	
	Now, suppose that $\alpha < -\beta$. If there is $\bar{t}>0$ such that $x(\bar{t}) >\mu_2$ then, by reasoning in the same manner as in Case~3, we get
	$$x(t)= \left(x(\bar{t})-1\right)e^{-\frac{1}{\eta}(t-\bar{t})}+1$$
	for all $t\geq \bar{t}$. It follows that $\lim\limits_{t\to\infty}x(t)=1$, i.e., $x(t)$ converges to the KKT point $\bar{x}=1$. If such a value $\bar{t}>0$ does not exist, then  $x(t) \leq\mu_2$ for all $t>0$. The continuity of $x(\cdot)$ implies the existence of $\tau>0$ such that $\mu_1< x(t)\leq \mu_2$ for all $t\in[0,\tau]$.  Thus, by utilizing~\eqref{dot_xt} and~\cite[Formula~(3.43), p.53]{Teschl_2012}, one can show that $x(t)$ is of the form~\eqref{ode_sol_2} for all $t\in[0,\tau]$. In addition, recalling that $\alpha < -\beta$, we have
	\begin{equation*}\label{eq2}
		\mu_2+\dfrac{\beta}{\alpha} =	\dfrac{\beta+\rho}{\rho-\alpha} +\dfrac{\beta}{\alpha} =\dfrac{\rho(\alpha+\beta)}{\alpha(\rho-\alpha)}> 0.
	\end{equation*}  Therefore,
	$$x(t)=\left(x^0+\dfrac{\beta}{\alpha}\right)e^{\frac{-\alpha}{\eta\rho}t} -\dfrac{\beta}{\alpha}= \left(\mu_2+\dfrac{\beta}{\alpha}\right)e^{\frac{-\alpha}{\eta\rho}t} -\dfrac{\beta}{\alpha} > \left(\mu_2+\dfrac{\beta}{\alpha}\right)-\dfrac{\beta}{\alpha} = \mu_2,$$
	for all $t\in(0,\tau]$. This is impossible because $x(t)\leq \mu_2$ for all $t>0$. We have thus proved that $x(t)$ converges to the KKT point $\bar{x}=1$ as $t\to\infty$.
	
	{\bf Case 5:} $ \mu_1< x^0 < \mu_2$.  In this case, by the continuity of $x(\cdot)$ we can find a value $\tau >0$ such that $\mu_1 < x(t)< \mu_2$ for all $t\in[0,\tau]$. From~\eqref{dot_xt} and~\cite[Formula~(3.43), p.~53]{Teschl_2012}, it follows that the expression~\eqref{ode_sol_2} is valid for all $t\in[0,\tau]$. Since the behavior of the function $x(t)$ in~\eqref{ode_sol_2} when $t$ varies on the half-line $[0,\infty)$ depends on the sign of the coefficient $x^0+\dfrac{\beta}{\alpha}$ of the exponential term $e^{\frac{-\alpha}{\eta\rho}t}$, we consider the following subcases.
	
	$\quad$ {\bf Subcase 5a:} $x^0+\dfrac{\beta}{\alpha}=0$, i.e., $x^0 = -\dfrac{\beta}{\alpha}$. Setting $\bar{x}=-\dfrac{\beta}{\alpha}$, we have
	$\alpha\bar{x}+\beta =0$.
	So, by~\eqref{fermat_rule_1}, $\bar{x}= -\dfrac{\beta}{\alpha}$ is a KKT point of~\eqref{trusted_reg_sub_1}.  As $x^0\in [-1,1]$, we have $ -\dfrac{\beta}{\alpha}\in[-1,1]$. Since the formula~\eqref{ode_sol_2} is valid for any $t\in[0,\tau]$, we have $x(t)=-\dfrac{\beta}{\alpha}$ for every $t\in[0,\tau]$. Arguing similarly as in the proof of Claim 1, we can deduce that $x(t)=-\dfrac{\beta}{\alpha}$ for every $t\geq 0$. Thus,  $\lim\limits_{t\to\infty}x(t)=-\dfrac{\beta}{\alpha},$
	hence $x(t)$ tends to the KKT point $\bar{x}$  as $t\to\infty$.
	
	$\quad${\bf Subcase 5b:} $x^0 < -\dfrac{\beta}{\alpha}$. As $x^0\geq -1$ and $\alpha<0$, this inequality yields $\alpha < \beta$.  Then, $-\alpha +\beta >0$, and hence, by~\eqref{fermat_rule_1} we see that $\bar{x}:=-1$ is a KKT point of~\eqref{trusted_reg_sub_1}. The set $\Theta$ of all positive numbers $\theta$ satisfying 
	\begin{equation}\label{ine1}
		\mu_1 \leq \left(x^0+\dfrac{\beta}{\alpha}\right) e^{\frac{-\alpha}{\eta\rho}\theta} -\dfrac{\beta}{\alpha}\leq \mu_2
	\end{equation}
	is nonempty. Indeed, since $\mu_1 < x(t)< \mu_2$ for every $t\in[0,\tau]$ and $x(t)$ is of the form~\eqref{ode_sol_2}, taking any $\theta\in(0,\tau]$, we have~\eqref{ine1} because $$x(\theta) = \left(x^0+\dfrac{\beta}{\alpha}\right) e^{\frac{-\alpha}{\eta\rho}\theta} -\dfrac{\beta}{\alpha}.$$ For any $\theta\in\Theta$, from~\eqref{ine1} it follows that
	$$\dfrac{\beta-\rho}{\rho-\alpha} +\dfrac{\beta}{\alpha} \leq \left(x^0+\dfrac{\beta}{\alpha}\right) e^{\frac{-\alpha}{\eta\rho}\theta}\leq \dfrac{\beta+\rho}{\rho-\alpha} +\dfrac{\beta}{\alpha}.$$ 
	Since $x^0+\dfrac{\beta}{\alpha}<0$, multiplying the last expression by $\left(x^0+\dfrac{\beta}{\alpha}\right)^{-1}$ yields
	\begin{equation}\label{ine2}
		\dfrac{\rho(\beta-\alpha)}{(\rho-\alpha)(\alpha x^0+\beta)} \geq e^{\frac{-\alpha}{\eta\rho}\theta} \geq \dfrac{\rho(\beta+\alpha)}{(\rho-\alpha)(\alpha x^0+\beta)}.
	\end{equation}
	Thus, taking logarithm of both sides of the first inequality in~\eqref{ine2}, we get $\bar{\tau}\geq \theta$, where
	$$ \bar{\tau}:=-\dfrac{\eta\rho}{\alpha}\ln\left(\dfrac{\rho(\beta-\alpha)}{(\rho-\alpha)(\alpha x^0+\beta)}\right).$$
	For any $t > 0$, if the vector $x(t)=\left(x^0+\dfrac{\beta}{\alpha}\right) e^{\frac{-\alpha}{\eta\rho}t}-\dfrac{\beta}{\alpha}$ belongs to the segment $[\mu_1,\mu_2]$, then the value $\theta:=t$ satisfies the condition~\eqref{ine1}. This means that $t\in\Theta$. Hence, $t \leq \bar{\tau}$. Therefore, for each $t>\bar{\tau}$, either $x(t) <\mu_1$, or $x(t)>\mu_2$. Suppose for a while that there exist $t_1, t_2>\bar{\tau}$ such that $x(t_1) < \mu_1$ and $x(t_2) >\mu_2$. We may assume that $t_1<t_2$. By the intermediate value theorem (see e.g.~\cite[Theorem~4.23]{Rudin_1976}), for any $\lambda\in (\mu_1, \mu_2)$ there is some $\xi \in (t_1,t_2)$ such that $x(\xi) = \lambda$. We have arrived at a contradiction. Thus, either $x(t) <\mu_1$ for all $t>\bar{\tau}$ or $x(t) >\mu_2$ for all $t>\bar{\tau}$. If the first situation occurs, then take any $\bar{t}>\bar{\tau}$ and find $\tau_0 >\bar{\tau}$ such that $x(t) < \mu_1$ for all $t\in [\bar{t},\tau_0]$. Repeating the arguments used in Case~1, we can show that $x(t)$ converges to the KKT point $\bar{x}=-1$ of~\eqref{trusted_reg_sub_1} as $t\to\infty$. If the second situation occurs, then the inequality $\mu_2<x(t)$ holds for every $t>\bar{\tau}$. As $x(t) \leq 1$ for all $t\geq 0$, we have
	$$\dfrac{\beta+\rho}{\rho-\alpha}=\mu_2<1.$$
	Thus, $\beta <-\alpha$  or, equivalently, $\alpha+\beta<0$. By~\eqref{fermat_rule_1}, this implies that $\tilde{x}:=1$ is a KKT of~\eqref{trusted_reg_sub_1}. Arguing similarly as in Case~3 gives $\lim\limits_{t\to\infty} x(t)=\tilde{x}$.
	
	$\quad${\bf Subcase 5c:} $x^0 > -\dfrac{\beta}{\alpha}$. As $x^0<1$ and $\alpha<0$,  this inequality forces $\alpha<-\beta$. Hence, $\alpha+\beta <0$ and, in accordance with~\eqref{fermat_rule_1},  $\tilde{x}:=1$ is a KKT point of~\eqref{trusted_reg_sub_1}. Similarly as in Subcase~5b, let $\Theta$ be the set of $\theta>0$ satisfying~\eqref{ine1}. Then, one has
	$$\dfrac{\rho(\beta-\alpha)}{(\rho-\alpha)(\alpha x^0+\beta)} \leq e^{\frac{-\alpha}{\eta\rho}\theta} \leq \dfrac{\rho(\beta+\alpha)}{(\rho-\alpha)(\alpha x^0+\beta)}.$$
	It follows from the second inequality of the last expression that
	$$\theta\leq -\dfrac{\eta\rho}{\alpha}\ln\left(\dfrac{\rho(\beta+\alpha)}{(\rho-\alpha)(\alpha x^0+\beta)}\right)=:\hat{\tau}.$$
	For any $t > 0$, if the vector $x(t)=\left(x^0+\dfrac{\beta}{\alpha}\right) e^{\frac{-\alpha}{\eta\rho}t}-\dfrac{\beta}{\alpha}$ is contained in the segment $[\mu_1,\mu_2]$, then the value $\theta:=t$ satisfies the condition~\eqref{ine1}. This implies that $t\in\Theta$. As a consequence, we have $t \leq \hat{\tau}$.  Arguing analogously as in Subcase~5b, we find that either $x(t) <\mu_1$ for all $t>\hat{\tau}$ or $x(t) >\mu_2$ for all $t>\hat{\tau}$. If $x(t) <\mu_1$ for all $t>\hat{\tau}$, it follows from the fact $x(t)\geq -1$ for all $t>\hat{\tau}$ that $$\dfrac{\beta-\rho}{\rho-\alpha}=\mu_1 >-1.$$ Hence, $\beta >\alpha$, or equivalently, $-\alpha+\beta >0$. By~\eqref{fermat_rule_1}, this implies that $\bar{x}:=-1$ is a KKT point of~\eqref{trusted_reg_sub_1}. Therefore, arguing similarly as in Case~1, we can infer that $x(t)$ tends to $\bar{x}=-1$ as $t\to\infty$. If $x(t) > \mu_2$ for all $t>\hat{\tau}$, then following the reasoning in Case~3, we see that $x(t)$ converges to the KKT point $\tilde{x}=1$ as $t\to\infty$. 
	
	By combining the results from Cases 1--5, we arrive at the desired conclusion.
    $\hfill\Box$
\end{proof}

Another partial answer to Question~7 is provided by the next proposition, which deals with convex quadratic programs and the corresponding monotone affine variational inequalities.

\begin{proposition}\label{prop2:dim1}
	For any $(\alpha,\beta)\in\R\times\R$ satisfying $\alpha\geq0$ and for any $x^0\in [-1,1]$, the unique solution $x(t)$ of~\eqref{dynamic_sys_A1}  converges to a KKT point of~\eqref{trusted_reg_sub_1} as $t\to\infty$, provided that $\rho-\alpha\geq 0$. 
\end{proposition}

\begin{proof}
	Let $x^0\in [-1,1]$ be given arbitrarily and let $x(t)$ be the unique solution of~\eqref{dynamic_sys_A1}. According to Theorems~\ref{global_sol_A} and~\ref{flow_invariant_A}, the trajectory $x(t)$ is defined on the whole real line and $x(t)\in C$ for all $t\geq 0$.
	
	By the assumptions of the proposition, $\alpha\geq0$. Hence, either $\alpha=0$ or $\alpha>0$. If $(\alpha,\beta)=(0,0)$, then one has $C_*=[-1,1]$, where $C_*$ denotes the KKT point of~\eqref{trusted_reg_sub_1}. For every $t\in [0,\infty)$, since $x(t)-\dfrac{1}{\rho}\big(\alpha x(t)+\beta\big)=x(t)$ and $x(t)\in C$, the differential equation in~\eqref{dynamic_sys_A1} becomes $\dot{x}(t)=0$ for all $t>0$. Thus, by~\eqref{dynamic_sys_A1} we have $x(t)=x^0$ for all $t\geq 0$. As $x^0\in C_*$, the assertion saying that $x(t)$ converges to a KKT point of~\eqref{trusted_reg_sub_1} as $t\to\infty$ is valid. To prove this fact in the general case, we cannot apply many arguments of the above proof, because the number $(\rho-\alpha)$ can be negative or null, as well. However, a modified version of that proof scheme can be used.
	
	For every $x\in \R$, put $u(x)=x-\dfrac{1}{\rho}(\alpha x+\beta)$ and observe that
	\begin{itemize}
		\item[\rm (i)] $u(x)<-1$ if and only if $(\rho-\alpha)x<\beta-\rho$;
		\item[\rm (ii)] $u(x)>1$ if and only if $(\rho-\alpha)x>\beta+\rho$;
		\item[\rm (iii)]  $-1\leq u(x) \leq 1$  if and only if $\beta-\rho\leq (\rho-\alpha)x\leq\beta+\rho$.
	\end{itemize}
	Therefore, thanks to~\eqref{proj_u} we can equivalently rewrite the differential relation in~\eqref{dynamic_sys_A1} as follows
	\begin{equation}\label{dot_xt_n}
		\dot{x}(t)=\begin{cases}
			\dfrac{1}{\eta}\big(-1-x(t)\big)\qquad &\mbox{\rm if}\;\; (\rho-\alpha)x(t)<\beta-\rho,\\[2ex]
			\dfrac{1}{\eta}\big(1-x(t)\big) &\mbox{\rm if}\;\; (\rho-\alpha)x(t)>\beta+\rho,\\[2ex]
			-\dfrac{1}{\eta\rho}\big(\alpha x(t) +\beta\big)&\mbox{\rm if}\;\;\beta-\rho\leq (\rho-\alpha)x(t)\leq\beta+\rho.
		\end{cases}
	\end{equation}

	\smallskip
By our assumption, $\rho-\alpha\geq 0$. The assertion will be proved by analyzing the following~3 cases. Observe that, in some sense, the second case is symmetric to the first one. In addition, note that if none of the situations in the first two cases occurs, then the third case must happen.
	
	\smallskip
     $\quad$ {\bf Case 1:}  \textit{There exists $\bar t\geq 0$ such that \begin{equation}\label{ine_c1} (\rho-\alpha)x(\bar t)<\beta-\rho.\end{equation}}
     
   	 \smallskip 	{\sc Claim 1:} \textit{The unique solution $x(t)$ of~\eqref{dynamic_sys_A1} converges to $\bar x:=-1$.}
	
	\smallskip 
	Indeed, by~\eqref{ine_c1} and the continuity of $x(\cdot)$ we can find some $\tau >0$ such that 
	\begin{equation}\label{tau} (\rho-\alpha)x(t)<\beta-\rho\quad\mbox{for all} \;\;t\in [\bar t,\bar t+\tau].\end{equation}
	Hence, from~\eqref{dot_xt_n} we get the following initial value problem for a first-order linear ordinary differential equation with constant coefficients
	\begin{equation*}\label{dynamic_sys_A1_c1_n}
		\begin{cases}
			\dot{x}(t)= -\dfrac{1}{\eta}x(t)-\dfrac{1}{\eta}, \quad\mbox{\rm for all}\;\; t\in [\bar t,\bar t+\tau],\\
			x(\bar t)=\theta
		\end{cases}
	\end{equation*} with $\theta:=x(\bar t)$. So, from the result in~\cite[Formula~(3.43), p.~53]{Teschl_2012} it follows that
	\begin{equation}\label{ode_sol_1_n0}
		x(t)=\left(\theta+1\right)e^{-\frac{1}{\eta}(t-\bar t)}-1
	\end{equation}
	for all $t\in [\bar t,\bar t+\tau]$. As $\theta+1\geq 0$, applied for all $t\in [\bar t,\infty)$, the formula 
	\begin{equation}\label{ode_sol_1_n}
		y(t)=\left(\theta+1\right)e^{-\frac{1}{\eta}(t-\bar t)}-1
	\end{equation} defines a non-increasing function $y:[\bar t,\infty)\to\mathbb R$. From~\eqref{ode_sol_1_n} it follows that
	\begin{equation}\label{dynamic_sys_A1_c1_nb}
		\begin{cases}
			\dot{y}(t)= -\dfrac{1}{\eta}y(t)-\dfrac{1}{\eta}, \quad\mbox{\rm for all}\;\; t\in [\bar t,\infty),\\
			y(\bar t)=\theta.
		\end{cases}
	\end{equation}
	Moreover,
	$$(\rho-\alpha) y(t)\leq (\rho-\alpha) y(\bar t)=(\rho-\alpha) x(\bar t)< \beta-\rho\quad\mbox{for all} \;\;t\in [\bar t,\infty).$$ Hence, for every $t\in [\bar t,\infty)$, one has $u(y(t))<-1$. It follows that 
	$$\dfrac{1}{\eta}\left[P_C(u(y(t)))-y(t)\right]= -\dfrac{1}{\eta}y(t)-\dfrac{1}{\eta}.$$ Combining this with~\eqref{dynamic_sys_A1_c1_nb}, we can assert that the condition
	$$\dot{y}(t)= \dfrac{1}{\eta} \left[P_C\left(y(t)-\dfrac{1}{\rho}\big(\alpha y(t)+\beta\big)\right)-y(t)\right]$$ is satisfied for all $t\in [\bar t,\infty)$. Therefore, recalling that $y(\bar t)=\theta=x(\bar t)$ and the trajectory $x(t)$ under our consideration is unique, we get $x(t)=y(t)$  for all $t\in [\bar t,\infty)$. Hence, from~\eqref{dynamic_sys_A1_c1_nb} we obtain $\lim\limits_{t\to\infty} x(t)=-1$.
	
	 \smallskip
	{\sc Claim 2:} \textit{Condition~\eqref{ine_c1} implies that $\bar{x}=-1$  is a KKT point of~\eqref{trusted_reg_sub_1}.}
	
	\smallskip 
	Indeed, from~\eqref{ine_c1} we can deduce that $(\rho-\alpha)(x(\bar t)+1)<\beta-\alpha$. Since $\rho-\alpha\geq 0$ and $x(\bar t)+1\geq 0$, this yields $\beta-\alpha>0$. So, $\bar{x}=-1$ satisfies the condition~\eqref{fermat_rule_1}, i.e., $\bar{x}$ is a KKT point of~\eqref{trusted_reg_sub_1}. 
	
	\smallskip
	From Claims~1 and~2 it follows that the unique solution $x(t)$ of~\eqref{dynamic_sys_A1}  converges to a KKT point of~\eqref{trusted_reg_sub_1} as $t\to\infty$.
	
\medskip
{\bf Case 2:} \textit{There exists $\bar t\geq 0$ such that \begin{equation}\label{ine_c2} (\rho-\alpha)x(\bar t)>\beta+\rho.\end{equation}}

\smallskip 	{\sc Claim 3:} \textit{The unique solution $x(t)$ of~\eqref{dynamic_sys_A1} converges to $\bar x:=1$.}

\smallskip 
Indeed, using~\eqref{ine_c2} and the continuity of $x(\cdot)$ we can find some $\tau >0$ such that 
$$(\rho-\alpha)x(t)>\beta-\rho\quad\mbox{for all} \;\;t\in [\bar t,\bar t+\tau].$$
Hence, from~\eqref{dot_xt_n} we get the following initial value problem for a differential equation 
\begin{equation*}\label{dynamic_sys_A1_c1_n2}
	\begin{cases}
		\dot{x}(t)= -\dfrac{1}{\eta}x(t)+\dfrac{1}{\eta}, \quad\mbox{\rm for all}\;\; t\in [\bar t,\bar t+\tau],\\
		x(\bar t)=\theta
	\end{cases}
\end{equation*} with $\theta:=x(\bar t)$. Applying~\cite[Formula~(3.43), p.~53]{Teschl_2012} yields
\begin{equation}\label{x(t)_Case2}
	x(t)=\left(\theta-1\right)e^{-\frac{1}{\eta}(t-\bar t)}+1
\end{equation}
for all $t\in [\bar t,\bar t+\tau]$. Since $\theta-1\leq 0$, the function $y:[\bar t,\infty)\to\mathbb R$ defined by the formula 
\begin{equation}\label{ode_sol_1_n2}
	y(t)=\left(\theta-1\right)e^{-\frac{1}{\eta}(t-\bar t)}+1
\end{equation} for $t\in [\bar t,\infty)$ is non-decreasing . Clearly,~\eqref{ode_sol_1_n2} yields
\begin{equation}\label{dynamic_sys_A1_c1_nb2}
	\begin{cases}
		\dot{y}(t)= -\dfrac{1}{\eta}y(t)+\dfrac{1}{\eta}, \quad\mbox{\rm for all}\;\; t\in [\bar t,\infty),\\
		y(\bar t)=\theta.
	\end{cases}
\end{equation}
Moreover, by~\eqref{ine_c2} we have
$$(\rho-\alpha) y(t)\geq (\rho-\alpha) y(\bar t)=(\rho-\alpha) x(\bar t)> \beta-\rho\quad\mbox{for all} \;\;t\in [\bar t,\infty).$$ Then, the inequality $u(y(t))>1$ holds for every $t\in [\bar t,\infty)$. Therefore, 
$$\dfrac{1}{\eta}\left[P_C(u(y(t)))-y(t)\right]= -\dfrac{1}{\eta}y(t)+\dfrac{1}{\eta}.$$ So,~\eqref{dynamic_sys_A1_c1_nb2} assures that the condition
$$\dot{y}(t)= \dfrac{1}{\eta} \left[P_C\left(y(t)-\dfrac{1}{\rho}\big(\alpha y(t)+\beta\big)\right)-y(t)\right]$$ is fulfilled for all $t\in [\bar t,\infty)$. Hence, the uniqueness of the trajectory $x(t)$ and the condition $y(\bar t)=\theta=x(\bar t)$ imply that $x(t)=y(t)$  for all $t\in [\bar t,\infty)$. Hence, from~\eqref{ode_sol_1_n2} we can deduce that $\lim\limits_{t\to\infty} x(t)=\bar x$.

\smallskip
{\sc Claim 4:} \textit{Condition~\eqref{ine_c2} implies that $\bar{x}:=1$  is a KKT point of~\eqref{trusted_reg_sub_1}.}

\smallskip 
Indeed, from~\eqref{ine_c2} we can deduce that $(\rho-\alpha)(x(\bar t)-1)>\beta+\alpha$. Since $\rho-\alpha\geq 0$ and $x(\bar t)-1\leq 0$, this yields $\beta+\alpha<0$. So, the condition~\eqref{fermat_rule_1} is fulfilled for $\bar{x}=1$; hence, $\bar{x}\in C_*$.  

\smallskip
From Claims~3 and~4 we get the desired conclusion on the asymptotic behavior of the trajectory $x(t)$ of~\eqref{dynamic_sys_A1}.  

\smallskip
{\bf Case 3:} \textit{For any $t\geq 0$, one has \begin{equation}\label{ine_c3} \beta-\rho\leq (\rho-\alpha)x(t)\leq\beta+\rho.\end{equation}}

\smallskip 
Thanks to~\eqref{dot_xt_n}, we can use condition~\eqref{ine_c3} to equivalently rewrite the system~\ref{dynamic_sys_A1} as
\begin{equation}\label{dynamic_sys_A1_n}
	\begin{cases}
		\dot{x}(t)= -\dfrac{1}{\eta\rho}\big(\alpha x(t) +\beta\big), \quad\mbox{\rm for all}\;\; t>0,\\
		x(0)= x^0.
	\end{cases}
\end{equation}

If $\alpha=0$, then~\eqref{dynamic_sys_A1_n} implies that \begin{equation}\label{special_traj} x(t)=x^0-\dfrac{\beta}{\eta\rho}t, \quad\mbox{\rm for all}\;\; t\geq 0.\end{equation} Since the situation $(\alpha,\beta)=(0,0)$ has been considered at the beginning of this proof, we can assume that $\beta\neq 0$. In this case, the trajectory $x(t)$ given by~\eqref{special_traj} is unbounded. This comes in conflict with~\eqref{ine_c3}, which now reads $$\rho^{-1}(\beta-\rho)\leq x(t)\leq \rho^{-1}(\beta+\rho)$$ for all $t\geq 0$. Thus, under assumption made in this subcase, if $\alpha=0$, then one must have $\beta=0$, and the desired asymptotic behavior of the trajectory $x(t)$ is guaranteed.

If $\alpha\neq 0$, then by~\cite[Formula~(3.43), p.~53]{Teschl_2012} we get
\begin{equation}\label{ode_sol_2_n}
	x(t)=	\left(x^0+\dfrac{\beta}{\alpha}\right)e^{\frac{-\alpha}{\eta\rho}t} -\dfrac{\beta}{\alpha}, \quad\mbox{\rm for all}\;\; t\geq 0.
\end{equation}

\smallskip
{\sc Claim 5:} \textit{If $\alpha\neq 0$, then the point $\bar x:=-\dfrac{\beta}{\alpha}$ belongs to $C=[-1,1]$.}

Indeed, if $\rho=\alpha$, then~\eqref{ine_c3} yields $\beta-\alpha\leq 0\leq\beta+\alpha$. Hence, $-1\leq -\dfrac{\beta}{\alpha}\leq 1$. Now, suppose that $\rho\neq \alpha$. Since $\rho-\alpha\geq 0$, we must have $\rho-\alpha>0$. Then, combining~\eqref{ine_c3} with~\eqref{ode_sol_2_n} gives
\begin{equation}\label{ine_c3_new} \dfrac{\beta-\rho}{\rho-\alpha}\leq \left(x^0+\dfrac{\beta}{\alpha}\right)e^{\frac{-\alpha}{\eta\rho}t} -\dfrac{\beta}{\alpha}\leq \dfrac{\beta+\rho}{\rho-\alpha}\end{equation} for all $t\geq 0$. Passing the relations~\eqref{ine_c3_new} to limits as $t\to\infty$, we get
\begin{equation*}\label{ine_c3_n} \dfrac{\beta-\rho}{\rho-\alpha}\leq -\dfrac{\beta}{\alpha}\leq \dfrac{\beta+\rho}{\rho-\alpha},\end{equation*} which implies that $-1\leq -\dfrac{\beta}{\alpha}\leq 1$. Thus, $\bar x\in C$.

Since $f(x)=\dfrac{1}{2} \alpha x^2 +\beta x$, the number $\bar x=-\dfrac{\beta}{\alpha}$ is a solution of the equation $\nabla f(x)=0$. As $\bar x\in C$, this implies that $\bar x$ is a KKT point of~\eqref{trusted_reg_sub_1}. Finally, from~\eqref{ode_sol_2_n} we can deduce that $\lim\limits_{t\to\infty} x(t)=\bar x$.

The proof is complete. $\hfill\Box$
\end{proof}

\begin{remark} The condition $\rho-\alpha\geq 0$ is not essential for the analysis of Case~3 in the above proof, hence it can be dropped in Case~3. Indeed, if $\rho-\alpha <0$, then from~\eqref{ine_c3} and~\eqref{ode_sol_2_n} we have
	\begin{equation}\label{ine_c3_new1} \dfrac{\beta-\rho}{\rho-\alpha}\geq \left(x^0+\dfrac{\beta}{\alpha}\right)e^{\frac{-\alpha}{\eta\rho}t} -\dfrac{\beta}{\alpha}\geq \dfrac{\beta+\rho}{\rho-\alpha}\end{equation} for all $t\geq 0$. Letting $t\to\infty$, from~\eqref{ine_c3_new1} we obtain
	\begin{equation*}\dfrac{\beta-\rho}{\rho-\alpha}\geq -\dfrac{\beta}{\alpha}\geq \dfrac{\beta+\rho}{\rho-\alpha},\end{equation*} which implies that $-1\leq -\dfrac{\beta}{\alpha}\leq 1$.
\end{remark}

\begin{remark} The condition $\rho-\alpha\geq 0$ is essential for the validity of Claims 1--4 in the proof of Proposition~\eqref{prop2:dim1} or for some related proof arguments. To justify this remark, consider the quadratic optimization problem~\eqref{trusted_reg_sub_1} with $\alpha=2$ and $\beta=1$. Then, we have $f(x)=x^2 +x$. So,~\eqref{trusted_reg_sub_1} is a strongly convex quadratic program having the unique solution $\bar x=-\frac{1}{2}$. Hence, $C_*=\{-\frac{1}{2}\}$. Choose $\eta=\rho=1$, $x^0=1$, and let $x(t)$ be the unique solution of~\eqref{dynamic_sys_A1}. Here we have $\rho-\alpha=-1$. For $\bar t:=0$, it holds that $$(\rho-\alpha)x(\bar t)=-1<0=\beta-\rho.$$ Hence, condition~\eqref{ine_c1} is satisfied. Formula~\eqref{ode_sol_1_n0} now becomes the expression $
x(t)=2e^{-t}-1$ for all $t\in [0,\tau]$, provided that $\tau>0$ be such that the requirement~\eqref{tau} is fulfilled. Since the latter now reads
\begin{equation*}\label{tau_e} x(t)>0\quad\mbox{for all} \;\;t\in 0,\tau] \end{equation*} and since the function $
x(t)=2e^{-t}-1$ is decreasing on $[0,\infty)$, we must have $\tau<-\ln (\frac{1}{2})$. Hence, the expression $
x(t)=2e^{-t}-1$ does not hold for all $t\in [0,\infty)$. Therefore, unlike the final part of the proof of Claim 1, the action of taking limit of this expression of $x(t)$ as $t\to\infty$ cannot be used here. Since $C_*=\{-\frac{1}{2}\}$, Claim~2 is invalid, despite the fact that condition~\eqref{ine_c1} is satisfied. Similarly, choosing $x^0=-1$ and $\bar t=0$, we can show that  the action of taking limit of the expression~\eqref{x(t)_Case2} of $x(t)$ as $t\to\infty$ cannot be used to justify Claim~3. Again, since $C_*=\{-\frac{1}{2}\}$, Claim~4 is invalid, although condition~\eqref{ine_c2} is fulfilled.
\end{remark}

Combining the results in Proposition~\ref{prop1:dim1} and Proposition~\ref{prop2:dim1}, we obtain the following theorem, which answers Question~7 in the affirmative under an additional assumption on the parameter $\rho$ when $\alpha>0$.

\begin{theorem}\label{Thm4.1}
	For any $(\alpha,\beta)\in\R\times\R$ and for any $x^0\in [-1,1]$, the unique solution $x(t)$ of~\eqref{dynamic_sys_A1}  converges to a KKT point of~\eqref{trusted_reg_sub_1} as $t\to\infty$ if either $\alpha\leq 0$ or $\alpha>0$ and $\rho\geq\alpha$. 
\end{theorem}

\begin{remark} Specializing the result in Theorem~\ref{Thm4.1} for the case $\eta=1$, we see that for any $(\alpha,\beta)\in\R\times\R$ and for any $x^0\in [-1,1]$, the unique solution $x(t)$ of the dynamical system
	\begin{equation*}
		\begin{cases}
			\dot{x}(t)= P_C\left(x(t)-\dfrac{1}{\rho}\big(\alpha x(t)+\beta\big)\right)-x(t), \quad\mbox{\rm for all}\;\; t>0,\\
			x(0)= x^0.
		\end{cases}
	\end{equation*}
  converges to a KKT point of~\eqref{trusted_reg_sub_1} as $t\to\infty$  if either $\alpha\leq 0$ or $\alpha>0$ and $\rho\geq\alpha$. This fact follows from~\cite[Theorem~1]{Antipin_94}, where no condition on the positive parameter $\rho$ was needed. 
\end{remark} 

\vskip 6mm
\noindent{\bf Acknowledgments}

\noindent A major part of the results of this paper was obtained at Dipartimento di Informatica, Universit\`a di Pisa, Pisa, Italy, in September 2024. N.~N.~Thieu and N.~D.~Yen would like to thank Professor M.~Pappalardo for his invitation, effective research cooperation, and very warm hospitality.

\end{document}